\documentclass[11pt,leqno,amscd,amssymb,pstricks,verbatim]{amsart}
\usepackage{amsfonts,amssymb}
\usepackage{amsmath,amscd}

\renewcommand{\subsection}[1]{\vspace{.18in}\par\noindent\addtocounter{subsection}{1}\setcounter{equation}{0}{\bf\thesubsection.\hspace{5pt}#1}}

\theoremstyle{definition}
\newtheorem{Def}[subsection]{Definition}
\newtheorem{Rem}[subsection]{Remark}

\theoremstyle{plain}
\newtheorem{Prop}[subsection]{Proposition}
\newtheorem{Thm}[subsection]{Theorem}
\newtheorem{Lem}[subsection]{Lemma}
\newtheorem{Coro}[subsection]{Corollary}
 \numberwithin{equation}{subsection}

\newcommand{\wt}{\mathrm{wt}}

\newcommand{\Image}{\operatorname{Im}}
\newcommand{\End}{\operatorname{End}}
\newcommand{\Hom}{\operatorname{Hom}}

\newcommand{\spann}{\operatorname{span}}
\newcommand{\diag}{\operatorname{diag}}
\newcommand{\height}{\operatorname{ht}}

\newcommand{\bin}{\bigcup}

\newcommand{\han}{\subseteq}

\def\fS{{\frak S}}

\def\fD{{\frak D}}

\newcommand{\klan}{\bfh_{\la,n}}
\newcommand{\kmun}{\bfh_{\mu,n}}
\newcommand{\knun}{\bfh_{\nu,n}}
\newcommand{\klanp}{\bfh_{\la^*,n}}

\newcommand{\bfMod}{{\bf Mod}}

\newcommand{\leb}{\left[}

\newcommand{\rib}{\right]}

\newcommand{\etad}{\eta^\dashv}

\newcommand{\Sr}{\ti\bfV(\iy,r)}
\newcommand{\tiN}{\ti\sN}

\newcommand{\mbzeta}{\mathbb Z^{\eta}}
\newcommand{\mbneta}{\mathbb N^{\eta}}
\newcommand{\mbni}{\mathbb N^{\infty}}
\newcommand{\mbzi}{\mathbb Z^{\infty}}

\newcommand{\mbznn}{\mathbb Z^{[-n,n]}}

\newcommand{\mbnnn}{\mathbb N^{[-n,n]}}
\newcommand{\mbnnno}{\mathbb N^{[-n,n-1]}}

\def\hmod{\text{-\bf mod}}

\newcommand{\Laetar}{\Lambda(\eta,r)}

\newcommand{\Lair}{\Lambda(\infty,r)}
\newcommand{\tiTheta}{\widetilde\Xi(\eta)}

\newcommand{\tiThmn}{\widetilde\Xi([m,n])}
\newcommand{\tiThi}{\widetilde\Xi(\infty)}
\newcommand{\Thetapm}{\Xi^{\pm}(\eta)}
\newcommand{\Thetap}{\Xi^{+}(\eta)}
\newcommand{\Thetam}{\Xi^{-}(\eta)}
\newcommand{\Thetar}{\Xi(\eta,r)}
\newcommand{\Thet}{\Xi(\eta)}

\newcommand{\Thip}{\Xi^+(\infty)}
\newcommand{\Thim}{\Xi^-(\infty)}
\newcommand{\Thiz}{\Xi^0(\infty)}
\newcommand{\Thir}{\Xi(\infty,r)}
\newcommand{\Thipm}{\Xi^\pm(\infty)}

\newcommand{\Thnnpm}{\Xi^\pm([-n,n])}
\newcommand{\Thmnpm}{\Xi^\pm([m,n])}

\newcommand{\Thnnopm}{\Xi^\pm([-n,n-1])}

\newcommand{\Thmn}{\Xi([m,n])}

\newcommand{\Lannr}{\Lambda([-n,n],r)}

\newcommand{\tiThmnz}{\widetilde\Xi([m_0,n_0])}

\newcommand{\bfUn}{\bfU([-n,n])}
\newcommand{\bfUnr}{\bfU([-n,n],r)}
\newcommand{\bfUeta}{\bfU(\eta)}

\newcommand{\bfUetar}{\bfU(\eta,r)}
\newcommand{\Uetar}{U(\eta,r)}
\newcommand{\bfVn}{\bfV([-n,n])}
\newcommand{\bfVnr}{\bfV([-n,n],r)}
\newcommand{\bfVnrpz}{\bfV^{\geq 0}([-n,n],r)}
\newcommand{\bfVnrmz}{\bfV^{\leq 0}([-n,n],r)}
\newcommand{\bfVmn}{\bfV([m,n])}
\newcommand{\bfVeta}{\bfV(\eta)}
\newcommand{\bfVi}{\bfV(\iy)}
\newcommand{\bfVir}{\bfV(\iy,r)}

\newcommand{\bfVetar}{\bfV(\eta,r)}

\def\sA{{\mathcal A}}
\def\sB{{\mathcal B}}
\def\sC{{\mathcal C}}

\def\sF{{\mathcal F}}

\def\sH{{\mathcal H}}

\def\sK{{\mathcal K}}

\def\sM{{\mathcal M}}
\def\sN{{\mathcal N}}
\def\sO{{\mathcal O}}

\def\sR{{\mathcal R}}
\def\sS{{\mathcal S}}
\def\sT{{\mathcal T}}

\def\sZ{{\mathcal Z}}
\def\leqwt{{\leq_{\text{wt}}}}
\def\geqwt{{\geq_{\text{wt}}}}

\newcommand{\diU}{\dot{\bfU}(\infty)}

\newcommand{\mcT}{\mathcal T}
\newcommand{\lann}{\lambda_{[-n,n]}}

\newcommand{\Kir}{\sK(\infty,r)}
\newcommand{\bfKir}{\bfK(\infty,r)}
\newcommand{\Ketar}{\sK(\eta,r)}
\newcommand{\bfKetar}{\bfK(\eta,r)}
\newcommand{\hbfKetar}{\widehat{\bfK}(\eta,r)}
\newcommand{\hbfKir}{\widehat{\bfK}(\iy,r)}

\newcommand{\tibfKir}{\h\bfK^\dagger(\iy,r)}
\newcommand{\tibfKirl}{\!^\dagger\h\bfK(\iy,r)}
\newcommand{\nKrb}{\bfK([-n,n],r)}

\newcommand{\Keta}{\sK(\eta)}
\newcommand{\bfKeta}{\bfK(\eta)}
\newcommand{\hbfKeta}{\widehat\bfK(\eta)}
\newcommand{\hbfKi}{\widehat\bfK(\iy)}

\newcommand{\tibfKi}{\h\bfK^\dagger(\iy)}
\newcommand{\tibfKil}{\!^\dagger\h\bfK(\iy)}

\newcommand{\iKt}{\h{\bfK}^\dagger(\infty)}
\newcommand{\tzr}{\tilde{\zeta}_r}
\newcommand{\AL}{^\dagger\!\h\sA}
\newcommand{\AR}{\h\sA^\dagger}

\newcommand{\bfKi}{\bfK(\infty)}
\newcommand{\iKh}{\widehat{\bfK}(\infty)}

\newcommand{\Jr}{J_r}
\newcommand{\Jrp}{J_{r+1}}
\newcommand{\nWmu}{W([-n,n],\mu)}

\newcommand{\iWmu}{W(\infty,\mu)}

\newcommand{\CrU}{\bfU(\iy,r)\text{-}{\bf mod}}
\newcommand{\CrS}{\bfS(\iy,r)\text{-}{\bf mod}}
\newcommand{\CrK}{\bfKir\text{-}{\bf mod}}
\newcommand{\CrUp}{\bfU(\iy,r')\text{-}{\bf mod}}
\newcommand{\polC}{{\mathcal C^{pol}}}
\newcommand{\intC}{\mathcal C^{int}}
\newcommand{\Cr}{{\mathcal C}_r}
\newcommand{\Cro}{{\mathcal C}_{r+1}}

\newcommand{\iX}{X(\infty)}
\newcommand{\iXp}{X^+(\infty)}
\newcommand{\iPi}{\Pi(\infty)}

\newcommand{\Xnnp}{X^+([-n,n])}

\newcommand{\iU}{\bfU(\infty)}
\newcommand{\iUr}{\bfU(\infty,r)}
\newcommand{\iSr}{\bfS(\infty,r)}
\newcommand{\iVr}{\bfV(\infty,r)}

\newcommand{\iUm}{\bfU^-(\infty)}

\newcommand{\amnA}{{}_a(A^{^{[m,n]}})}
\newcommand{\amnB}{{}_a(B^{^{[m,n]}})}
\newcommand{\amnC}{{}_a(C^{^{[m,n]}})}
\newcommand{\mnA}{A^{^{[m,n]}}}

\newcommand{\tte}{\mathtt{e}}
\newcommand{\ttf}{\mathtt{f}}
\newcommand{\ttk}{\mathtt{k}}
\newcommand{\tth}{\mathtt{h}}
\newcommand{\ttm}{\mathtt{m}}
\newcommand{\ttn}{\mathtt{n}}

\def\leq{\leqslant}\def\geq{\geqslant}
\def\le{\leqslant}\def\ge{\geqslant}
 \newcommand{\iy}{\infty}

 \newcommand{\Og}{\Omega}

 \def\fkf{{\frak f}}

\newcommand{\Th}{\Xi}
 \newcommand{\dt}{\delta}
 \newcommand{\Dt}{\Delta}
 \newcommand{\lm}{\longmapsto}
 \newcommand{\map}{\mapsto}
 \newcommand{\vp}{\varpi}

 \newcommand{\bfOg}{{\bf\Omega}}
 \newcommand{\bfOgir}{\bfOg(\iy,r)}

 \newcommand{\og}{\omega}
 
  \newcommand{\vi}{\varphi}
 \newcommand{\st}{\stackrel}
 
 \newcommand{\up}{\upsilon}
 
 \newcommand{\al}{\alpha}
 \newcommand{\bt}{\beta}
 \newcommand{\h}{\widehat}
 \newcommand{\ti}{\widetilde}
 \newcommand{\dzr}{\dot{\zeta}_r}

\newcommand{\zr}{\zeta_r}
\newcommand{\bfxir}{\xi_r}

\newcommand{\bfzetar}{\zeta_r}
\newcommand{\barbfzetar}{\varsigma_r}

 \newcommand{\s}{\sigma}
 
 \newcommand{\p}{\prec}
 \newcommand{\pr}{\preccurlyeq}
 \newcommand{\bop}{\bigoplus}
 \newcommand{\op}{\oplus}
 \newcommand{\ot}{\otimes}
 \newcommand{\bfK}{\boldsymbol{\mathcal K}}

 \newcommand{\bfk}{\mathbf{k}}
 \newcommand{\bfl}{\mathbf{0}}

 \newcommand{\bfe}{\mathbf{e}}
 \newcommand{\bff}{\mathbf{f}}
 \newcommand{\mc}{\mathcal}

 \newcommand{\lra}{\longrightarrow}
 \newcommand{\ra}{\rightarrow}
 \newcommand{\la}{\lambda}
 \newcommand{\La}{\Lambda}

\newcommand{\eap}{{\tte}^{(A^+)}}
\newcommand{\Eap}{E^{(A^+)}}
\newcommand{\faf}{{\ttf}^{(A^-)}}
\newcommand{\Faf}{F^{(A^-)}}

 \newcommand{\eA}{e_{_{A}}}
 \newcommand{\eB}{e_{_{B}}}
 \newcommand{\eC}{e_{_{C}}}
 \newcommand{\gABC}{g_{_{A,B,C}}}

 \newcommand{\mbn}{\mathbb N}
 \newcommand{\mbq}{\mathbb Q}
 \newcommand{\mbz}{\mathbb Z}

 \newcommand{\bfm}{{\mathbf{m}}}

 \newcommand{\bfj}{{\mathbf{j}}}

 \newcommand{\bfh}{{\mathbf{h}}}
\newcommand{\bft}{{\mathbf t}}

 \newcommand{\bfU}{{\mathbf{U}}}
 \newcommand{\bfV}{{\mathbf{V}}}

\newcommand{\bfH}{{\boldsymbol{\mathcal H}}}
\newcommand{\bfS}{{\boldsymbol{\mathcal S}}}

\newcommand{\be}{\boldsymbol{e}}

\def\bH{{\boldsymbol{\mathcal H}}}
\def\bS{{\boldsymbol{\mathcal S}}}
\def\bfsi{{\boldsymbol\sigma}}

\begin{document}
\title{Quantum $\frak {gl}_\infty$, infinite  $q$-Schur algebras and their representations}
\author{Jie Du and
Qiang Fu$^\dagger$}
\address{School of Mathematics and Statistics, University of New South Wales,
Sydney 2052, Australia.}
\address{{\it Home page:} \tt
http://web.maths.unsw.edu.au/$\sim$jied} \email{j.du@unsw.edu.au}
\address{Department of Mathematics, Tongji University, Shanghai, 200092, China.}
\email{q.fu@hotmail.com}

\thanks{$^\dagger$Corresponding author.}

\thanks{Supported by the Australian Research Council (Grant: DP
0665124) and partially by the National Natural Science Foundation
of China (10601037 \& 10671142). The paper was written while the
second author was visiting the University of New South Wales.
}

\begin{abstract} In this paper, we investigate the structure and
representations of the quantum group $\iU=\mathbf
U_\up(\frak{gl}_\iy)$. We will present a realization for $\iU$,
following Beilinson--Lusztig--MacPherson (BLM) \cite{BLM}, and
show that the natural algebra homomorphism $\zr$ from $\iU$ to the
infinite $q$-Schur algebra $\iSr$ is not surjective for any $r\geq
1$. We will give a BLM type realization for the image
$\iUr:=\zr(\iU)$ and discuss its presentation in terms of
generators and relations. We further construct a certain
completion algebra $\tibfKi$ so that $\zr$ can be extended to an
algebra epimorphism $\ti\zr:\tibfKi\to\iSr$. Finally we will
investigate the representation theory of $\iU$, especially the
polynomial representations of $\iU$.
\end{abstract}
 \sloppy \maketitle
\begin{center}
{\it Dedicated to Professor Leonard L. Scott on the occasion of his 65th
birthday}
\end{center}

\section{Introduction}
The Lie algebra $\frak{gl}_\iy$ of infinite matrices and its
extension $A_\iy$ are interesting topics in the theory of infinite
dimensional Lie algebras. Certain highest weight representations
of these algebras have important applications in finding solutions
of a large class of nonlinear equations (see, e.g., \cite{DJKM})
and in the representations theories of Heisenberg algebras, the
Virasoro algebra and other Kac-Moody algebras (see, e.g.,
\cite{Kac81,Kac87,Pa1,Pa2}). The quantum versions of the
corresponding universal enveloping algebras have also been studied
in \cite{LS91,PS97,PS98}. It should be noted that the structure
and representations of quantum $\frak{gl}_\iy$ have various
connections with the study of Lie superalgebras $\frak{gl}(m|n)$
(see, e.g., \cite{Br}) and the study of quantum affine
$\frak{gl}_n$ (see \cite[\S6]{Ha}, \cite[\S 2]{MM}, \cite{DGr} and
\cite{Mc}).

In this paper, we will investigate the quantum group $\mathbf
U(\infty)=\mathbf U_v(\frak{gl}_\iy)$ and its representations
through a series of its quotient algebras $\mathbf U(\iy,r)$ and
the infinite quantum Schur algebras $\iSr$. The main idea is to
extend the approach developed in \cite{BLM} for quantum $\frak
{gl}_n$ to quantum $\frak {gl}_\infty$. Thus, we obtain a
realization for $\iU$ and an explicit description of the algebra
homomorphism $\zeta_r:\iU\to\mathbf U(\iy,r)$. This in turn gives
rise to a presentation and various useful bases for $\mathbf
U(\iy,r)$. We will also prove that $\mathbf U(\iy,r)$ is a proper
subalgebra of $\iSr$. This fact shows that the classical
Schur--Weyl duality fails in this case. Finally, we investigate
the `polynomial' representation theory and classify all
irreducible polynomial representations for $\iU$. We expect that
this work will have further applications to quantum affine $\frak
{gl}_n$.

We organize the paper as follows.
We recall the definition
of quantum $\frak{gl}_{\eta}$ and $q$-Schur algebras $\sS(\eta,r)$
at any consecutive segment $\eta$ of $\mbz$ in \S 2. In \S 3, we
use the $\eta$-step flag variety to define the algebra
$\sK(\eta,r)$ and discuss its stabilization property developed in
\cite{BLM} in a context suitable for infinite $\eta$. At the end
of \S3, we will focus on the infinite case $\eta=\mbz$. First, the
stabilization property allows us to define an algebra
$\boldsymbol\sK(\infty)$ over $\mbq(\up)$ whose completion
$\h{\boldsymbol\sK}(\infty)$ contains a subalgebra $\mathbf
V(\iy)$ which is isomorphic to $\iU$. Second, there is a algebra
homomorphism $\xi_r:\mathbf V(\iy)\to
\h{\boldsymbol\sK}(\infty,r)$ with image $\bfVir$. A Drinfeld--Jimbo
type presentation
for $\bfVir$ will be given in \S 4. In \S 5, we will establish
isomorphisms $\iSr\cong {\h{\boldsymbol\sK}}^\dagger(\infty,r)$
and between $\bfVir$ and the homomorphic image $\iUr$ of the
natural homomorphism $\zr:\iU\to\iSr$. Thus, we conclude that
$\zr$ is not surjective for any $r\geq 1$. In \S 6, by identifying
$\boldsymbol\sK(\iy)$ with the modified quantum group $\diU$, we
derive an algebra epimorphism
$\dot\zr:\boldsymbol\sK(\iy)\to\boldsymbol\sK(\iy,r)$ and hence,
extend the map $\zr$ to an epimorphism from $\tibfKi$ to $\iSr$.
In the last three sections, we investigate the representation
theory of $\iU$. The highest weight representations of $\iU$ is
studied in \S 7, and the polynomial representations of $\iU$  is
in \S 9 as an application of the representation theory of $\iSr$
investigated in \S 8.

Throughout the paper, we will encounter several (associative)
algebras $\sA$ over a commutative ring $\sR$ without the identity
element, but with many orthogonal idempotents $e_i,\,\,i\in I,$
such that $\sA=\oplus_{i,i'\in I}e_i\sA e_{i'}$ and, for all
$i,i'\in I$, $e_i\sA e_{i'}$ are free over $\sR$. 
Clearly, the index set $I$ must be an infinite set, since if $I$
was a finite set, then $\sum_{i\in I}e_i$ would be the identity
element of $\sA$.

Choose an $\sR$-basis $\sB=\{a_j\}_{j\in J}$ such that  $e_i\sB
e_{i'}=\{e_ia_je_{i'}\}_{j\in J}\backslash\{0\}$ is a basis for
$e_i\sA e_{i'}$ for all $i,i'\in I$, and $\sB=\cup_{i,i'}e_i\sB
e_{i'}$. We further assume that $e_i\in\sB$ for all $i\in I$.
 It is clear that,  for any $j\in J$, there exist unique
$i_j,i_j'\in I$ such that $e_{_{i_j}}a_j=a_j$ and
$a_je_{_{i_j'}}=a_j$. We will write $ro(j):=i_j$ and $co(j):=i_j'$
for all $j\in J$.

For a formal infinite linear combination $f=\sum_{j\in J}f_ja_j$
with $f_j\in \sR$, let, for any $i\in I$, $J(e_i,f):=\{j\in J\mid
f_j\neq0, e_ia_j=a_j\}=\{j\in J\mid f_j\neq0, ro(j)=i\}$ and
$J(f,e_i):=\{j\in J\mid f_j\neq0, a_je_i=a_j\}=\{j\in J\mid
f_j\neq0, co(j)=i\}.$ One can easily check the following lemma.

\begin{Lem}\label{completion}
 Let 
 $\frak L$ be the set\footnote{We may identify the set $\frak L$
 with the direct product $\prod_{j\in J}\sR a_j$.} of formal (possibly infinite)
linear combinations of $\sB$, and let
\begin{equation*}
\begin{split}
\AL&:=\{f\in\frak L\mid \forall i\in I,|J(e_i,f)|<\infty\},
\quad\AR:=\{f\in\frak L\mid \forall i\in I, |J(f,e_i)|<\infty\},\\
\widehat\sA&:=\{f\in\frak L\mid \forall i,i'\in I,
|J(e_i,f)|<\infty,\,|J(f,e_{i'})|<\infty\}=\AL\cap\AR.
\end{split}
\end{equation*} We define the product of two elements
 $\sum_{s\in J}f_sa_s,\sum_{t\in J}g_ta_t$ in $\AR$ (resp., $\AL$) to be $\sum_{s,t\in J}f_sg_ta_sa_t$ where $a_sa_t$ is the product in $\sA$.
 Then $\AL$ and $\AR$  become
associative algebras with identity $1=\sum_{i\in I}e_i$ and
$\widehat\sA$ is a subalgebra with the same identity.
\end{Lem}

The algebras $\AL$, $\AR$ and $\widehat\sA$ are called the {\it
completion algebras} of $\sA$.

Note that, if $\sA$ admits an algebra antiautomorphism $f$
satisfying $f(e_i)=e_i$ for all $i\in I$, then
$\AL\cong(\AR)^{\text{op}}$, the algebras with the same underlying
space as $\AR$ but opposite multiplication.

\vspace{0.3cm} \noindent{\bf Some notations and conventions.}
A
{\it consecutive segment} $\eta$ of $\mbz$ is either a finite
interval of the form $[m,n]:=\{i\in\mbz\mid m\le i\le n\}$ or an
infinite interval of the form $(-\infty,m]$, $[n,+\infty)$ and
$(-\infty,\infty)=\mbz$, where $m,n\in\mbz$. Let
$$
\eta^\dashv=\begin{cases} \eta\backslash\{\eta_{\max}\},&\text{ if
there is
a maximum element $\eta_{\max}$ in $\eta$};\\
\eta,& \text{ otherwise.}\\
\end{cases}$$
 Let $M_\eta(\mbz)$ (resp. $\mbz^\eta$) be the set of
all matrices $(a_{i,j})_{i,j\in\eta}$   (resp. all sequences
$(a_i)_{i\in\eta}$) over $\mbz$ with {\it finite support} if $\eta$ is infinite.
We will use the following index sets throughout the paper.
\begin{equation}\label{notations}
\aligned
\ti\Xi(\eta)&=\{(a_{ij})\in M_\eta(\mbz)\mid
a_{ij}\ge0\,\,\forall i\neq j\text{ in }\eta\}=\Thetam+\widetilde\Xi^0(\eta)+\Thetap,\\
\Xi(\eta)&=\{(a_{ij})\in M_\eta(\mbz)\mid
a_{ij}\ge0\,\,\forall i, j\in \eta\}=\Thetam+\Xi^0(\eta)+\Thetap,\\
\Thetapm&=\{(a_{ij})\in M_\eta(\mbz)\mid
a_{ii}=0\,\,\forall i\in\eta\}=\Thetam+\Thetap,\\
\Xi(\eta,r)&=\{A\in\Xi(\eta)\mid \s(A):=\Sigma_{i,j\in\eta}a_{ij}=r\},\,\, \text{ and }\\
\Laetar&=\{\la\in\mbneta\mid \s(\la):=\Sigma_{i\in\eta}\lambda_i=r\}
\endaligned
\end{equation}
where $\Thetap$ (resp., $\Thetam$, $\widetilde\Xi^0(\eta)$) is the subset
of $\tilde\Xi(\eta)$ consisting of those matrices $(a_{ij})$ with $a_{ij}=0$
for all $i\geq j$ (resp., $i\leq j$, $i\neq j$), $\Xi(\eta)= \widetilde\Xi^0(\eta)\cap\Xi(\eta)$,
and $\mbn^\eta\subset\mbz^\eta$. The sum $\Thetam+\widetilde\Xi^0(\eta)+\Thetap$ indicates the following matrix decomposition:
for $A\in\ti\Xi(\eta)$, $A=A^++A^0+A^-$ with $A^+\in\Xi^+(\eta)$,
$A^-\in\Xi^-(\eta)$, $A^0\in\widetilde\Xi^0(\eta)$. We also write $A^\pm=A^++A^-\in\Xi^{\pm}(\eta)$.
Also, for $A=(a_{ij})\in M_\eta(\mathbb N)$, define
\begin{equation}\label{degA}\deg(A):=\sum_{i,j\in\eta}|j-i|a_{ij}.\end{equation}

Throughout the paper, $\up$ denotes an indeterminate and $\mathbb Q(\up)$ denotes the
fraction field of the integral Laurent polynomial ring $\mathcal Z:=\mathbb Z[\up,\up^{-1}]$.

\section{Quantum $\frak{gl}_{\eta}$ and $q$-Schur algebras at $\eta$}

Let $\eta$ be a fixed consecutive segment of $\mbz$.

\begin{Def}\label{definition of U(infty)}
The quantum ${\frak {gl}}_\eta$ over $\Bbb Q(\up)$ is the $\Bbb
Q(\up)$-algebra $\bfU(\eta):=\bfU_v({\frak {gl}}_\eta)$ presented
by generators
$$E_i,\ F_i\quad(i\in\eta^\dashv),\ K_j,\ K_j^{-1}\quad(j\in\eta)$$
and relations

$(a)\ K_{i}K_{j}=K_{j}K_{i},\ K_{i}K_{i}^{-1}=1;$

$(b)\ K_{i}E_j=\upsilon^{\dt_{i,j}-\dt_{i,j+1}} E_jK_{i};$

$(c)\ K_{i}F_j=\upsilon^{\dt_{i,j+1}-\dt_{i,j}} F_jK_i;$

$(d)\ E_iE_j=E_jE_i,\ F_iF_j=F_jF_i\ when\ |i-j|>1;$

$(e)\ E_iF_j-F_jE_i=\delta_{i,j}\frac
{\widetilde K_{i}-\widetilde K_{i}^{-1}}{\upsilon-\upsilon^{-1}},\
where \ \widetilde K_i =K_{i}K_{i+1}^{-1};$

$(f)\ E_i^2E_j-(\upsilon+\upsilon^{-1})E_iE_jE_i+E_jE_i^2=0\
 when\ |i-j|=1;$

$(g)\ F_i^2F_j-(\upsilon+\upsilon^{-1})F_iF_jF_i+F_jF_i^2=0\
 when\ |i-j|=1.$
\end{Def}
The algebra $\bfUeta$ is a Hopf algebra with comultiplication
$\Dt$ defined on generators by $ \Dt(E_i)=E_i\ot \ti K_i+1\ot
E_i$, $\Dt(F_i)=F_i\ot 1+\ti K_i^{-1}\ot F_i$, $\Dt(K_j)=K_j\ot
K_j.$
For notational simplicity, we set
$$\bfU(\eta)=\begin{cases}\bfU(n)=\bfU(\mathfrak{gl}_n),&\text{ if }\eta=[1,n];\\
\bfU(\infty)=\bfU(\mathfrak{gl}_\infty),&\text{ if }\eta=(-\infty,+\infty).\\\end{cases} $$

Let $\bfU^+(\eta)$ (resp., $\bfU^-(\eta)$, $\bfU^0(\eta)$) be the
subalgebra of $\bfUeta$ generated by the $E_i$ (resp., $F_i$,
$K_j^{\pm 1}$).  The subalgebras $\bfU^+(\eta)$ and $\bfU^-(\eta)$
are both $\mbn$-graded in terms of the degrees of monomials in the
$E_i$ and $F_i$. For monomials $M$ in the $E_i$ and $M'$ in the
$F_i$, and an element $h\in\bfU^0(\eta)$, write
$\deg(MhM')=\deg(M)+\deg(M')$. Note that $\deg$ does {\it not}
define an algebra grading on $\bfUeta$. However, if
$\Pi(\eta)=\{\al_j:=\be_j-\be_{j+1}\mid j\in\eta^\dashv\}$ denotes
the set of simple roots, where $\be_i=(\cdots,0,\underset
i1,0\cdots)\in\mbzeta,$ then there is an algebra grading over the
root lattice $\mbz\Pi(\eta)$,
\begin{equation}\label{alg grading}
\bfU(\eta)=\bop\limits_{\nu\in\mbz\Pi(\eta)}\bfU(\eta)_\nu
\end{equation} defined by the conditions
$\bfU(\eta)_{\nu'}\bfU(\eta)_{\nu''}\han\bfU(\eta)_{\nu'+\nu''}$,
$K^\bfj\in\bfU(\eta)_0$, $E_i\in\bfU(\eta)_{\al_i}$,
$F_i\in\bfU(\eta)_{-\al_i}$ for all $\nu',\nu''\in\mbz\Pi(\eta)$,
$i\in\eta^\dashv$ and $\bfj\in\mbz^\eta$.

Let $[m]^{!}=[1][2]\cdots[m]$
where
$[t]=\frac{\upsilon^t-\upsilon^{-t}}{\upsilon-\upsilon^{-1}}\in\sZ$. The
Lusztig $\sZ$-form is the $\sZ$-subalgebra $U(\eta)$ of $\bfUeta$
generated by the elements
$E_i^{(m)}=\frac{E_i^m}{[m]^!},\,\,F_i^{(m)}=\frac{F_i^m}{[m]^!},$
$K_j$ and
\begin{equation}\label{binomialK}
 \bigg[ {K_j;c \atop t}
\bigg] = \prod_{s=1}^t \frac
{K_j\upsilon^{c-s+1}-K_j^{-1}\upsilon^{-c+s-1}}{\upsilon^s-\upsilon^{-s}},\end{equation}
for all $i\in\etad$, $j\in\eta$, $m,t\in\mathbb N$ and
$c\in\mathbb Z$. Let $U^+(\eta)$ (resp., $U^-(\eta)$, $U^0(\eta)$)
be the subalgebra of $U(\eta)$ generated by the $E_i^{(m)}$
(resp., $F_i^{(m)}$, $K_j^{\pm 1}$ and $\leb{K_j;c\atop t}\rib$).

For each $A\in\Thetapm$ and $\bfj=(j_i)_{i\in\mbz}\in\mbzeta$,
choose $m,n\in\eta$ such that $m\le n$ and $A\in\Thmnpm$ and let
\begin{equation*}
E^{(A^+)}=M_nM_{n-1}\cdots M_{m+1},\quad
F^{(A^-)}=M_{m+1}'M_{m+2}'\cdots M_n'\ \text{ and } \
K^\bfj=\prod_{i\in\mbz}K_i^{j_i},
\end{equation*}
where
$$M_j=E_{j-1}^{(a_{j-1,j})}(E_{j-2}^{(a_{j-2,j})}E_{j-1}^{(a_{j-2,j})})
\cdots(E_{m}^{(a_{m,j})}E_{m+1}^{(a_{m,j})}\cdots
E_{j-1}^{(a_{m,j})}),$$ and
$$M_j'=(F_{j-1}^{(a_{j,m})}\cdots
F_{m+1}^{(a_{j,m})}F_{m}^{(a_{j,m})})
\cdots(F_{j-1}^{(a_{j,j-2})}F_{j-2}^{(a_{j,j-2})})
F_{j-1}^{(a_{j,j-1})}.$$
 It is clear that $E^{(A^+)}$ and
 $F^{(A^-)}$ are independent of the selection of $m,n$.  Note that we
 have,  for $A\in\Thetapm$,
 $\deg(E^{(A^+)})=\deg(A^+)$ and $\deg(F^{(A^-)})=\deg(A^-)$.

For a finite $\eta$, the following result is due to Lusztig (see, e.g.,
\cite[2.14]{Lu901}, \cite[5.7]{BLM} and \cite{Du95}). The infinite case follows from
the algebra isomorphism
\begin{equation*}
\bfU(\infty)\cong\underset{\underset n\longrightarrow}\lim\,\bfU([-n,
n]),
\end{equation*}induced from the natural embeddings
$\bfU([-n,n])\subseteq\bfU([-n-1,n+1])\subseteq\bfU(\infty)$ for
all $n\ge 0$.

\begin{Prop} \label{Monomial base for U(infty)}
{\rm(1)} The set $\{\Eap K^\bfj\Faf\mid A\in\Thetapm,\mathbf
j\in\mbzeta\}$ forms a $\Bbb Q(\up)$-basis for $\bfUeta$.

{\rm(2)} The set
$\bigg\{\Eap\prod_{i\in\eta}K_i^{\dt_i}\leb{K_i;0\atop
t_i}\rib\Faf\ \big|\ A\in\Thetapm,\bft\in\mbneta,
\dt_i\in\{0,1\}\text{ for $i\in\eta$}\bigg\}$ forms a $\sZ$-basis
for $U(\eta)$.
\end{Prop}

We now extend the definition of $q$-Schur algebras (or quantum
Schur algebras). Let $\Og_\eta$ be a free $\sZ$-module with basis
$\{\og_i\}_{i\in\eta}$. Let $\bfOg_\eta=\Og_\eta\otimes\mbq(\up)$.
Then $\bfU(\eta)$ acts naturally on $\bfOg_\eta$ by
$K_a\og_b=\up^{\dt_{a,b}}\og_b\,(a,b\in\eta)$,
$E_a\og_b=\dt_{a+1,b}\og_a$ and
$F_a\og_b=\dt_{a,b}\og_{a+1}\,(a\in\eta^\dashv,b\in\eta).$ The
tensor space $\bfOg_\eta^{\ot r}$ is a $\bfU(\eta)$-module via the
comultiplication $\Dt$ on $\bfUeta$. Further, restriction gives
the $U(\eta)$-module $\Og_\eta^{\ot r}$.

Let $\sH$ be the Hecke algebra over $\sZ$ associated with the
symmetric group $\fS_r$, and let $\bH=\sH\otimes_\sZ\mbq(\up)$.
Thus, $\sH$ as a $\sZ$-algebra has a basis $\{\sT_w\}_{w\in\fS_r}$
subject the relations: for all $w\in \fS_r$ and $s\in
S:=\{(i,i+1)\mid 1\le i\le r-1\}$
$$\sT_s\sT_w=\begin{cases} \sT_{sw},&\text{ if }\ell(sw)>\ell(w);\\
(\up-\up^{-1})\sT_w+\sT_{sw},&\text{ if }\ell(sw)<\ell(w),
\end{cases}$$
where $\ell$ is the length function on $\fS_r$ with respect to
$S$.

The Hecke algebra $\sH$ acts on $\Og_\eta^{\otimes r}$ on the
right by ``place permutations'' via
   $$(\og_{i_1}\!\cdots \og_{i_r})\sT_{(j,j+1)}
     =\begin{cases}
      \og_{i_1}\!\cdots \og_{i_{j+1}}\og_{i_j}\!\cdots \og_{i_r},
            &\text{if }i_j\!<\!i_{j+1};\cr
      v \og_{i_1}\!\cdots \og_{i_r},
            &\text{if }i_j\!=\!i_{j+1};\cr
      (v\!-\!v^{-1})\og_{i_1}\!\cdots \og_{i_r}+\og_{i_1}\!\cdots \og_{i_{j+1}}\og_{i_j}\!\cdots \og_{i_r},
            &\text{if }i_j\!>\!i_{j+1};
      \end{cases}
      $$
cf. \cite[(14.6.4)]{DDPW}. The endomorphism algebras
$$\bS(\eta,r):=\End_\bH(\bfOg_\eta^{\otimes
r}),\qquad \sS(\eta,r):=\End_\sH(\Og_\eta^{\otimes r})$$ are
called {\it $q$-Schur algebras at $(\eta,r)$}. Note that
$\bS(n,r):=\bS(\eta,r)$ for $\eta=[1,n]$
 is usually called the {\it $q$-Schur algebra} of bidegree $(n,r)$. In general, if $n=|\eta|$ is finite,
  $\bS(\eta,r)\cong\bS(n,r)$ and $\bS(\eta,r)\cong\sS(\eta,r)\otimes_\sZ\mbq(\up)$. If
$\eta=\mbz$, $\bS(\infty,r):=\bS(\eta,r)$ is called an {\it
infinite $q$-Schur algebra}. We shall use similar notations for
their integral versions.

Since the $\bH$-action commutes with the action of $\bfU(\eta)$
(see, e.g., \cite{Du95}, \cite[\S2-d]{Br}), we obtain algebra homomorphisms
\begin{equation}\label{bfzeta1}
\bfzetar:\bfU(\eta)\longrightarrow \bS(\eta,r),\qquad
\zr|_{U(\eta)}:U(\eta)\longrightarrow\sS(\eta,r).
\end{equation}
Let $\bfUetar=\Image(\bfzetar)$ and
$\Uetar=\Image(\zr|_{U(\eta)})$. It is known that both maps are
surjective if $\eta$ is finite. However, this is not the case when
$\eta$ is infinite; see \ref{not onto} below. Thus, for
$\eta=(-\infty,\infty)$ both $\bfU(\infty,r)$ and $U(\infty, r)$
are proper subalgebras. We will investigate the relationship
between $\bfU(\infty,r)$ and $\bS(\infty,r)$ in \S5 via the BLM
type realizations for $\bfU(\infty)$ and $\bfU(\infty,r)$
discussed in the next section.

\section{The BLM realization $\mathbf V(\infty)$ of $\iU$ and its quotients $\mathbf V(\infty,r)$}

In this section, we review the geometric construction of the
$q$-Schur algebra and extend it to the infinite case.

Let $V$ be a vector space of dimension $r$ over a field $k$. An
$\eta$-{\it step flag} is a collection $\fkf=(V_i)_{i\in\eta}$ of
subspaces of $V$ such that $V_i\subseteq V_{i+1}$ for all
$i\in\eta^\dashv$ and $\bin_{i\in\eta}V_i=V$ and $V_i=0$ for
$i\ll0$ if $\eta$ has no minimum element. If $\eta$ has a minimum
element $\eta_{\min}$, let $V_{\eta_{\min}-1}=0$.

Let $\sF$ be the set of $\eta$-step flags. The group $G:=GL(V)$
acts naturally on $\sF$, and hence, diagonally on $\sF\times\sF$.
For $(\fkf,\fkf')\in\sF\times\sF$, where $\fkf=(V_i)_{i\in\eta}$
and $\fkf'= (V'_i)_{i\in\eta}$, we let
$a_{i,j}=\text{dim}(V_{i-1}+(V_i\cap
V_j'))-\text{dim}(V_{i-1}+(V_i\cap V_{j-1}')).$
  Then the map $(\fkf,\fkf')\map(a_{i,j})$ induces a bijection from the
  set of $G$-orbits on $\sF\times\sF$ to the set $\Thetar$.
Let $\sO_{A}\han\mc F\times\mc F$  be the $G$-orbit corresponding
to the matrix $A\in\Thetar$. When $k$ is a finite field of
$q$-elements, every orbit $\mc O_A$ is a finite set. Thus, for any
$A,B,C \in\Thetar$ and any fixed $(\fkf_1,\fkf_2)\in \mc O_{C}$,
the number $$g_{_{A,B,C;q}}:=\{\fkf\in\mc F\mid (\fkf_1,\fkf)\in\mc
O_{A}, (\fkf,\fkf_2)\in\mc O_{B}\}$$ is independent of the
selection of $(\fkf_1,\fkf_2)$. It is well-known that there exists
a polynomial $g_{_{A,B,C}}\in\mbz[\up^2]$
 such that
$g_{_{A,B,C}}|_{\up^2=q}=g_{_{A,B,C;q}}$ for any $q$.

Let $\Ketar$ be the free $\sZ$-module with basis $\{e_{A}\mid
A\in\Xi(\eta,r)\}$. By \cite[1.2]{BLM} there is a associative
$\sZ$-algebra structure on $\Ketar$ with multiplication $\eA\cdot
\eB=\sum_{C\in\Thetar}\gABC\eC$. Note that, if $\eta$ is finite,
then $\Ketar$ is isomorphic to the $q$-Schur algebra
$\sS(|\eta|,r)$ (see, e.g., \cite{Du95}). However, when $\eta$ is
infinite, the algebra $\Ketar$ has no identity element. We further
consider the basis $\{[A]\}_{A\in\Thetar}$ for $\sK(\eta,r)$,
where $[A]:=\up^{-d_{A}}\eA\,\,\,\text{ with
}\,\,\,d_{A}=\sum_{i\geq k, j<l}a_{ij}a_{kl}.$

For a finite segment $\eta$ of $\mbz$, let $I=I_{\eta}\in\Thet$
be the identity matrix of size $|\eta|$. Consider $A\in\tiThmnz$.
For any $a,m,n\in\mbz$ with $m\leq m_0$ and $n_0\leq n$, let $\amnA$
denote the matrix in $\Thmn$ obtained from $A$ by adding 0s at $(i,j)$ positions
with $m\le i<m_0$ or $n_0<j\le n$, and let
$\amnA=\mnA+aI\in\tiThmn$. Clearly, when $a$ is large enough,
$\amnA\in\Thmn$.  By \cite{BLM} one can easily show the following
generalized version of the {\it stabilization property}.
\begin{Thm}\label{stabilization property}
Let $\sZ_1$ be the subring of $\mbq(\up)[\up']$ generated by
$\prod_{1\leq i\leq t}\frac{\up^{-2(a-i)}\up'^2-1} {\up^{-2i}-1}$
and $\up^j$ with $a\in\mbz$, $t\geq 1$ and $j\in\mbz$. Let $\eta$
be a fixed consecutive segment of $\mbz$. For any
$A,B,C\in\tiTheta$ with $co(A)=ro(B)$, there exist finitely many nonzero
elements $f_{A, B , C }(\up,\up')\in\sZ_1$ and an integer $a_0\geq
1$ such that, if $A,B,C\in\tiThmnz$ for some $m_0\leq n_0$ in
$\mbz$ with $[m_0,n_0]\subseteq\eta$, then
$$[\amnA]\cdot[\amnB]=\sum_{ C \in\tiThmnz}f_{A, B , C }(\up,\up^{-a})[\amnC]$$
for all $m\leq m_0$, $n\geq n_0$ with $[m,n]\subseteq\eta$, and
$a\geq a_0$.
\end{Thm}

We now use the polynomials $f_{A, B , C }(\up,\up')$ given in
\ref{stabilization property} to define a new associative algebra.

For $C=(c_{i,j})\in M_\eta(\mbz)$ let
$ro(C)=\bigl(\sum_jc_{i,j}\bigr)_{i\in\eta}\in\mbz^\eta$ and
$co(C)=\bigl(\sum_ic_{i,j}\bigr)_{j\in\eta}\in\mbz^\eta $. Consider
the free $\sZ_1$-module with basis $\{A\mid A\in\tiTheta\}$. Define
a multiplication on this module by linearly extending the products
on basis elements:
\begin{equation}\label{AB}
A\cdot B :=\begin{cases}\sum\limits_{ C \in\tiTheta}f_{A, B , C }(\up,\up') C , &\text{if $co(A)=ro( B )$}\\
0&\text{otherwise}
\end{cases}
\end{equation}
Then we get an associative algebra over $\sZ_1$ which has no
identity element.

By specializing $\up'=1$, we get an associative $\sZ$-algebra
$\Keta$ with basis $[A]:=A\otimes 1$ $(A\in\tiTheta)$ in which the
product $[A]\cdot[ B ]$ is given by $\sum_{ C \in\tiTheta}f_{A, B
, C }(\up,1)[ C ]$, if $co(A)=ro( B )$, and it is zero, otherwise.

Let $\bfKeta=\Keta\ot_{\sZ}\mbq(\up)$ and
$\bfKetar=\Ketar\ot_{\sZ}\mbq(\up)$. The set $\{[\diag(\la)]\mid
\la\in\mbzeta\}$ (resp., $\{[\diag(\la)]\mid \la\in\Laetar\}$) is
a set of orthogonal idempotents of $\bfKeta$ (resp., $\bfKetar$). By
\ref{completion} we may construct the completion algebra
$\hbfKeta$ (resp., $\hbfKetar$) of $\bfKeta$ (resp., $\bfKetar$).
The element $\sum_{\la\in\mbzeta}[\diag(\la)]$ (resp.,
$\sum_{\la\in\Laetar}[\diag(\la)]$) is the identity element of
$\hbfKeta$ (resp., $\hbfKetar$). Note that if $\eta$ is a finite
set, then we have $\bfKetar=\hbfKetar$.

Given $r>0$, $A\in\Thetapm$ and
${\bf j}\in \mbzeta$, we define
\begin{equation}\label{the definition of A(j) and A(j,r)}
\begin{split}
A({\bf j},r)=A({\bf j},r)_\eta &=\sum_{\la\in\La(\eta,r-\sigma(A))}\upsilon^{\la\cdot\bfj}[A+\diag(\la)]\in{\hbfKetar},\\
A({\bf j})=A({\bf j})_\eta &=\sum_{\la\in\mbz^\eta}
\upsilon^{\la\cdot\bfj}[A+\diag(\la)]\in\widehat{\bfK}(\eta).
\end{split}
\end{equation}
where $\la\cdot\bfj=\sum_{i\in\mbz}\la_ij_i<\iy$. (Note that these
summations depend on $\eta$.) In particular,
$$\aligned
A(\bfl)&=\sum_{\la\in\mbz^\eta}
[A+\diag(\la)], \text{ for the zero vector $\bfl\in\mbzeta$,}\\
 0(\bf j)&=\sum_{\la\in\mbz^\eta}
\upsilon^{\la\cdot\bfj}[\diag(\la)], \text{ for the zero matrix $0\in\Thetapm$}.\endaligned$$
 Thus,  \eqref{AB} implies the following.

\begin{Lem}\label{the commute formula in the completion algebra of K(infty)}
Let $\la,\mu\in\mbzeta$. For $i\in\etad$ we have
$$\aligned
&E_{i,i+1}(\bfl)[\diag(\la)]=[E_{i,i+1}+\diag(\la-\be_{i+1})]=[\diag(\la+\al_i)] E_{i,i+1}(\bfl)\\
& E_{i+1,i}(\bfl)[\diag(\la)]=[E_{i+1,i}+\diag(\la-\be_i)]=[\diag(\la-\al_i)]
E_{i+1,i}(\bfl).\endaligned$$
\end{Lem}

Let $\bfVeta$ (resp. $\bfVetar$) be the subspace of $\hbfKeta$
(resp. $\hbfKetar$) spanned by the elements $A(\bfj)$ (resp.,
$A(\bfj,r)$) for $A\in\Thetapm$ and $\bfj\in\mbzeta$. It is clear
that the elements $A(\bfj)$ ($A\in\Thetapm$, $\bfj\in\mbzeta$)
form a $\mbq(\up)$-basis of $\bfVeta$. There are similar bases for
$\bfVetar$; see \cite{DFW} and Corollary \ref{Ajr basis} below.

When $\eta=[m,n]$ is {\it finite}, it is known from
\cite[5.7]{BLM} that $\bfVmn$ is a subalgebra of
$\h{\boldsymbol{\mc K}}([m,n])$ and there is an isomorphism
\begin{equation}\label{bfVn}
\bfVmn\cong \mathbf U([m,n])\,\,\,\text{ via }\,\,\,
 E_h\mapsto E_{h,h+1}(\bfl),\quad K^{\bfj}\mapsto 0(\bfj),\quad F_h\mapsto E_{h+1,h}(\bfl).
 \end{equation}
Moreover, $\bfK([m,n],r)=\bfV([m,n],r)=\h\bfK([m,n],r)$ is a
$q$-Schur algebra in this case.

When $\eta$ is {\it infinite}, we will see below that the
isomorphism $\bfU(\eta)\cong\bfV(\eta)$ continues to hold, but the
equalities $\bfK([m,n],r)=\bfV([m,n],r)=\h\bfK([m,n],r)$ are
replaced by a chain relation
$\bfKetar\subseteq\bfVetar\subseteq\hbfKetar$; see \ref{map zr}(2)
and \ref{tri}(1) below. Moreover, $\bfVetar$ is isomorphic to
$\bfU(\eta, r)$.

For simplicity, we will only consider the case when $\eta=\mbz$ in
the sequel.

\begin{Thm}\label{map zr}
$(1)$ $\bfVi$ is a subalgebra of $\hbfKi$. Moreover, the algebra $\iU$ is isomorphic
 to $\bfVi$ by sending $E_h$ to $E_{h,h+1}(\bfl)$, $F_h$ to $E_{h+1,h}(\bfl)$, and
 $K^\bfj$ to $0(\bfj)$ where $h\in\mbz$ and $\bfj\in\mbzi$.

$(2)$ There is an algebra homomorphism, {\rm the truncation map},
\begin{equation}\label{bfxir}
\bfxir:\mathbf
V(\infty)\longrightarrow\hbfKir
\end{equation} by sending $A(\bfj)$ to
$A(\bfj,r)$ for any $A\in\Thipm$ and $\bfj\in\mbzi$. In
particular, we have $\bfVir=\bfxir(\mathbf V(\infty))$ is a
subalgebra of $\hbfKir$.
\end{Thm}
\begin{proof}
Let $f$ be the injective linear map from $\bfVn$ to $\bfVi$ by
sending $A(\bfj)$ to $A(\bfj)_\infty$ for any $A\in\Thnnpm$ and
$\bfj\in\mbznn$. Let $X:=\{E_{h,h+1}(\bfl),E_{h+1,h},0(\bfj)\mid
-n\leq h<n,\bfj\in\mbznn\}$.  Since the formulas similar to
\cite[5.3]{BLM} hold in $\hbfKi$,  we have
$f(xA(\bfj))=f(x)f(A(\bfj))$ for any $x\in X$ and $A\in\Thnnpm$
and $\bfj\in\mbznn$. Since the algebra $\bfVn$ is generated by
$X$; see (\ref{bfVn}), it follows that $f$ is an algebra
homomorphism. Thus, $\bfVi$ is the direct limit of $\bfVn$ and,
hence, is a subalgebra of $\hbfKi$, proving (1). The statement (2)
can be proved similarly.
\end{proof}

We shall see in Proposition \ref{iSr as dagger} below that $\hbfKir$ identifies with
a subalgebra of $\bS(\infty,r)$. Thus, identifying $\iU$ with $\bfVi$ yields
the fact that the maps $\xi_r$ in \eqref{bfxir} and $\zeta_r$ in \eqref{bfzeta1} are identical.

The quotient algebras $\bfVir$ share many
properties with the usual $q$-Schur algebras. We first prove that
$\bfKir$ is a subalgebra of $\bfVir$.

For $i\in\mbz$, let
$$\ttk_i=0(\boldsymbol e_i,r),\,\,\, \tte_i=
E_{i,i+1}(\bfl,r)\,\,\,\text{ and }\,\,\,\ttf_i=
E_{i+1,i}(\bfl,r),$$
 and, for any $\bft\in\mbni$ and
$\bfj\in\mbzi$,  let
$$\ttk(\bft)=\prod_{i\in\mbz}[\ttk_{i};t_{i}]^!,\,\,\,\ttk_\bft=\prod_{i\in\mbz}\left[{\ttk_i;0
\atop t_i}\right]\,\,\,\text{ and }
 \ttk^{\bfj}=\prod_{i\in\mbz}\ttk_i^{j_i}$$  where
$[x;t]^!=(x-1)(x-\up)\cdots(x-\up^{t-1})$, $[x;0]^!=1$ and
$\left[{\ttk_i;0 \atop t_i}\right]$ is defined as in
\eqref{binomialK}. Since $\bft$ has finite support, the products
are well-defined. By Theorem \ref{map zr},
$\tte_i,\ttk_i,\ttf_i\,\,(i\in\mbz)$ generate $\iVr$. One can
easily get the following result by direct calculation; cf.
\cite{DG}, \cite{DP}.

\begin{Lem} \label{KKK}
\begin{enumerate}
\item We have $\ttk(\bft)=0$ for all $\bft\in\mbni$ with
$\s(\bft)>r$.

\item For any $\bft\in\mbni$, we have
$\ttk_{\bft}=\begin{cases} 0,&\text{ if }\s(\bft)>r;\\
{[\diag(\bft)],}&\text{ if }\s(\bft)= r.
\end{cases}
$ In particular, if $\la\in\La(\iy,r)$, then
$[\diag(\la)]=\ttk_\la\in\bfV(\iy,r)$ and hence
$1=\sum_{\la\in\Lair}\ttk_\la$. \item For any $\la\in\La(\iy,r)$,
we have $\ttk_i\ttk_{\la}=\up^{\la_i}\ttk_{\la}$ and
$\left[{\ttk_i;c\atop t}\right]\ttk_{\la}=\left[{\la_i+c\atop
t}\right]$.
\end{enumerate}
\end{Lem}

One can easily check the following commutator relations by direct
calculation; cf. \cite[3.4]{DG} and \cite[4.8,4.9]{DP}. Let $\leq$ be the partial order on $\Laetar$ defined by setting,
for $\la,\mu\in \Laetar$,
\begin{equation*}\label{order on Lair}\la\leq\mu\text{ if and only
if } \la_i\leq\mu_i\text{ for all }i\in\eta.\end{equation*}
For
$A\in\ti\Xi(\iy)$ let
\begin{equation}\label{bfsigma}
\bfsi(A)=(\s_i(A))_{i\in\mbz}, \text{ where }
\s_i(A)=a_{i,i}+\sum_{j< i}(a_{i,j}+a_{j,i}).\end{equation}

\begin{Prop}\label{alp}
Let $A\in\Thipm$ and $\la\in\La(\iy,r)$.

$(1)$ If $\la\geq\bfsi(A^+)$, then
${\tte}^{(A^+)}\ttk_\la=\ttk_{\la-co(A^+)+ro(A^+)}{\tte}^{(A^+)}$;
otherwise, ${\tte}^{(A^+)}\ttk_\la=0$.

$(2)$ If $\la\geq\bfsi(A^-)$, then
${\ttf}^{(A^+)}\ttk_\la=\ttk_{\la+co(A^-)-ro(A^-)}{\ttf}^{(A^-)}$;
otherwise, ${\ttf}^{(A^+)}\ttk_\la=0$.
\end{Prop}

For every $A=(a_{s,t})\in\ti\Xi(\eta)$ and $i\not=j$,  let $
\s_{i,j}(A)=\sum_{s\leq i;t\geq j}a_{s,t}$ for $i<j$, and
$\s_{i,j}(A)=\sum_{s\geq i;t\leq j}a_{s,t}$ for $i>j$. Define $B
\pr A$ if and only if $\s_{i,j}( B )\leq\s_{i,j}(A)$ for all
$i,j\in\eta$, $i\not=j$. Put $ B \p A$ if $ B \pr A$ and
$\s_{i,j}( B )<\s_{i,j}(A)$ for some $i\not=j$. For
$A\in\Xi(\infty)$ with $\nu=\bfsi(A)$, let
$$\ttm^{(A)}={\tte}^{(A^+)}\ttk_\nu{\ttf}^{(A^-)}.$$
In general, for $A\in\Xi^\pm(\infty)$ and $\la\in\La(\iy,r)$,
let $\ttm^{(A,\la)}={\tte}^{(A^+)}\ttk_\la{\ttf}^{(A^-)}.$ Since
the formula similar to \cite[5.5(c)]{BLM} hold in $\bfVir$, the
following result can be proved similar to the proof of \cite[5.5
and 5.6]{DP}.
\begin{Thm} \label{tri}
$(1)$ For any $A\in\Thir$,  we have the following equation in
$\bfV(\iy,r)$.
\begin{equation}\label{monomial base for Kr(infty)}
\ttm^{(A)}=[A]+\sum_{ B \in\Thir, B \p A}f_{ B ,A}[ B
]\quad(\mathrm{a\ finite\ sum})
\end{equation}
where $f_{ B ,A}\in\sZ$.  Hence, the elements $\ttm^{(A)}$,
$A\in\Thir$, form a $\sZ$-basis of $\Kir$, and
$\bfKir\han\bfV(\iy,r)$.

$(2)$  Suppose $\ttm^{(A,\lambda)}\not=0$ for some $A\in\Thipm$
and $\la\in\La(\iy,r)$. If there exists $D\in\Thiz$ such that
$co(A+D)=\la+co(A^-)-ro(A^-)$, then $\ttm^{(A,\la)}=\ttm^{(A+D)}$.
Otherwise,
\begin{equation}\label{monomial base for Kr(infty)cc}
\ttm^{(A,\lambda)}=\sum_{ B \in\Thir, B \prec A} f'_{ B ,A}\ttm^{(
B )}\quad (\mathrm{a\ finite\ sum}).
\end{equation}
where $f'_{ B ,A}\in\mathbb{Q}(\up)$.
\end{Thm}

We end this section by showing that Borel subalgebras of (finite
dimensional) $q$-Schur algebras are natural subalgebras of
$\bfVir$, a fact which plays a crucial role in the determination
of a presentation for $\bfVir$.

 Let $\bfV^{\geq 0}([-n,n],r)$ (resp.
$\bfV^{\leq 0}([-n,n],r)$) be the subalgebras of $\bfVnr$
generated by $\bfe_i'=E_{i,i+1}({\bf0},r)_{[-n,n]}$ (resp.
$\bff_i'=E_{i+1,i}({\bf0},r)_{[-n,n]}$) and
$\bfk_j'=0(\be_j,r)_{[-n,n]}$ with $i\in[-n,n-1]$ and
$j\in[-n,n]$. For $\la\in\Lannr$, let
$$\bfk'_{\la}=\prod\limits_{i=-n}^n\left[{\bfk'_i;0 \atop
\la_i}\right]\in\bfVnrpz\cap \bfVnrmz$$ where
$\bfk'_n=\up^r\bfk_{-n}^{\prime -1}\cdots\bfk_{n-1}^{\prime -1}$.

\begin{Prop}\label{map vin}
For $n\geq 1$ there are algebra monomorphisms (sending $1$ to $1$)
 $$\aligned \vi_n^{\geq
0}:&\bfVnrpz\ra\bfVir\quad\text{ satisfying }\quad
\bfe'_i\mapsto\tte_i,\,\bfk'_i\mapsto\ttk_i\quad(-n\leq i\leq
n-1)\\
\vi_n^{\leq 0}:&\bfVnrmz\ra\bfVir\quad\text{ satisfying
  }\quad\bff'_i\mapsto\ttf_i,\,\bfk'_i\mapsto\ttk_i\quad(-n\leq i\leq n-1).\endaligned$$
  Moreover, for any $\la\in\La([-n,n],r)$, we have
  $\vi_n^{\geq 0}(\bfk'_\la)=\vi_n^{\leq 0}(\bfk'_\la)=\tth_{\la,n}$, where
\begin{equation}\label{tth}
\tth_{\la,n}=\sum_{\mu\in\La(\iy,r)
\atop\mu_{-n}=\la_{-n},\cdots,\mu_{n-1}=\la_{n-1}}[\diag(\mu)].\end{equation}
\end{Prop}
\begin{proof}
By \cite[8.1]{DP} there is an algebra homomorphism $\vi_n^{\geq
0}:\bfVnrpz\ra\bfVir$ mapping $\bfe'_i$ to $\tte_i$ and $\bfk'_i$
to $\ttk_i$ for $-n\leq i\leq n-1$. By \eqref{monomial base for
Kr(infty)} and \cite[8.2]{DP} we have $\vi_n^{\geq 0}$ is
injective. The proof for $\vi_n^{\leq 0}$ is similar. It remains
to prove the last assertion. We have
$$\aligned
\vi_n^{\geq
0}(\bfk'_\la)&=\prod\limits_{i=-n}^{n-1}\leb{\ttk_i;0\atop
t_i}\rib\cdot\leb{\up^r\ttk_{-n}^{-1}\ttk_{-n+1}^{-1}\cdots\ttk_{n-1}^{-1};0\atop
\la_n}\rib\\&=\sum_{\mu\in\Lair}\left(\prod_{i=-n}^{n-1}\leb{\mu_i\atop\la_i}\rib\leb{r-\sum_{i=-n}^{n-1}\mu_i\atop
\la_n}\rib\right)[\diag(\mu)].\\\endaligned$$ Since
$\la\in\La([-n,n],r)$ we have
$\sum_{i=-n}^{n-1}(\mu_i-\la_i)+(r-\sum_{i=-n}^{n-1}\mu_i-\la_n)=r-\sum_{i=-n}^n\la_i=0$.
Hence, for $\mu\in\Lair$ we have
\begin{equation*}
\begin{split}
\prod_{i=-n}^{n-1}\leb{\mu_i\atop\la_i}\rib\leb{r-\sum_{i=-n}^{n-1}\mu_i\atop
\la_n}\rib\not=0& \Longleftrightarrow \mu_{i}\geq\la_i \text{ for
} -n\leq i\leq n-1
\text{ and } r-\sum_{i=-n}^{n-1}\mu_i\geq\la_n\\
&\Longleftrightarrow \mu_{i}=\la_i \text{ for } -n\leq i\leq n-1.
\end{split}
\end{equation*}
Hence,
$\prod_{i=-n}^{n-1}\leb{\mu_i\atop\la_i}\rib\leb{r-\sum_{i=-n}^{n-1}\mu_i\atop
\la_n}\rib=1$, and \eqref{tth} follows.
\end{proof}

\begin{Rem}
The algebra monomorphisms $\vi_n^{\geq 0}$ and $\vi_n^{\leq 0}$
can't be glued as an algebra homomorphism from $\bfVnr$ to
$\bfVir$ since the last relation in \cite[4.1]{DP} does not hold
in $\bfVir$.
\end{Rem}

\section{Drinfeld--Jimbo type presentation for $\iVr$}
We first record the
handy result \cite[(6.6.2)]{DFW}) for later use. See \eqref{notations} and \eqref{bfsigma}
for relevant notations.

\begin{Lem}\label{handy}
 For $A\in\Th^\pm([-n,n])$, let
$$\La_{n,r,A}=\{\la\in\La([-n,n],r)\mid
\la\geq\bfsi(A)\}\,\,\text{ and
}\,\,\La_{n,r,A}'=\{\bfj\in\mbnnno\mid\sigma(\bfj)+\s(A)\leq
r\}.$$ Then $|\La_{n,r,A}|=|\La_{n,r,A}'|$ and
$\det\left(\up^{\sum_{i=-n}^{n-1}\la_ij_i}\right)_{\la\in\La_{n,r,A},\bfj\in\La_{n,r,A}'}\neq
0.$
\end{Lem}

We are now ready to describe a presentation for $\bfVir$.
\begin{Def}\label{definition of S(infty,r)} Let $\Sr$ be the associative algebra over $\mbq(\up)$
generated by  the elements
\[\bfe_i,\bff_i, \bfk_i
\quad (i\in\mbz),\] subject to the relations:
\begin{itemize}
\item[(a)] $\bfk_i\bfk_j=\bfk_j\bfk_i$; \item[(b)]
$\prod_{i\in\mbz}[\bfk_i;t_i]^!=0$  for all $\bft\in\mbni$ with
$\s(\bft)=r+1$; \item[(c)] $\bfe_i\bfe_j=\bfe_j\bfe_i,
\bff_i\bff_j=\bff_j\bff_i\quad (|i-j|>1)$; \item[(d)]
$\bfe_i^2\bfe_j-(\up+\up^{-1})\bfe_i\bfe_j\bfe_i+\bfe_j\bfe_i^2=0\quad(|i-j|=1)$;
\item[(e)]
$\bff_i^2\bff_j-(\up+\up^{-1})\bff_i\bff_j\bff_i+\bff_j\bff_i^2=0\quad(|i-j|=1)$;
\item[(f)] $\bfk_i\bfe_j=v^{\dt_{i,j}-\dt_{i,j+1}}
\bfe_{j}\bfk_i,\quad\bfk_i\bff_j=v^{-\dt_{i,j}+\dt_{i,j+1}}
\bff_{j}\bfk_i$; \item[(g)]
$\bfe_i\bff_j-\bff_j\bfe_i=\dt_{i,j}\frac
{\ti\bfk_i-\ti\bfk_i^{-1}}{\up-\up^{-1}}$, where
$\ti\bfk_i=\bfk_i\bfk_{i+1}^{-1}$, $i\in\mbz$.
\end{itemize}
\end{Def}

Comparing \ref{definition of S(infty,r)} with \ref{definition of
U(infty)}, we see that there is an algebra epimorphism
$\bfU(\iy)\twoheadrightarrow \Sr$ satisfying $E_i\mapsto\bfe_i$,
$F_i\mapsto\bff_i$ and $K_j\mapsto\bfk_j$. In particular, for
$A\in\Th^{\pm}(\iy)$, let $\bfe^{(A^+)}$, $\bff^{(A^-)}$,
$\bfk^{\bfj}$, etc., be the images of $E^{(A^+)}$, $F^{(A^-)}$,
$K^{\bfj}$, etc., under this homomorphism.

On the other hand, by \ref{map zr}(2) and \ref{KKK}(1), there is
an algebra epimorphism \begin{equation}\label{pi}\pi:\Sr
\longrightarrow\bfVir,\end{equation} defined by
 $\bfe_i\mapsto\tte_i,\quad\bff_i\mapsto\ttf_i,\quad\bfk_i\mapsto\ttk_i.$ We shall prove that $\pi$ is an isomorphism
by displaying a spanning set for $\Sr$ whose image under $\pi$ is
a linearly independent set in $\bfVir$.

We first have the following counterpart of \ref{map vin} by
comparing the above defining relations with the relations for
Borel subalgebras.

\begin{Lem}For any given $n\geq 1$, there are algebra
monomorphisms
$$\aligned\psi_n^{\geq 0}:&\bfVnrpz\ra\Sr\quad\text{ satisfying }\quad
\bfe'_i\mapsto\bfe_i,\,\bfk'_i\mapsto\bfk_i\quad(-n\leq i\leq
n-1)\\
\psi_n^{\leq 0}:&\bfVnrmz\ra\Sr\quad\text{ satisfying }\quad
  \bff'_i\mapsto\bff_i,\,\bfk'_i\mapsto\bfk_i\quad(-n\leq i\leq n-1).\endaligned$$
\end{Lem}
\begin{proof}
The injectivity follows from \ref{map vin} since $\vi_n^{\geq0}$
is the composition of $\psi_n^{\geq0}$ and $\pi$.
\end{proof}

This result shows that the subalgebra ${\mathbf B}_n^+$ (resp.
${\mathbf B}_n^-$) of $\Sr$ generated by $\bfe_i,\bfk_i$ (resp.
$\bff_i,\bfk_i$), $-n\leq i\leq n-1$, is a (finite dimensional)
Borel subalgebra investigated in \cite[\S8]{DP}. Let ${\mathbf
B}^0_n={\mathbf B}^+_n\cap {\mathbf B}^-_n$. Then ${\mathbf
B}^0_n$ is the subalgebra generated by $\bfk_i$, $-n\leq i\leq
n-1$. We summarize the properties of these algebras; cf.
\cite[8.2--3]{DG} and \cite[4.7--10]{DP}.

For any $\la\in\La([-n,n],r)$, let $\klan:=\psi_n^{\geq
0}(\bfk'_\la)=\psi_n^{\leq 0}(\bfk'_\la)$.

\begin{Coro}\label{property2 of tilde(ttk)la} Let $n\geq 1$.
\begin{enumerate}
\item Each of the following sets forms a basis for ${\mathbf
B}_n^0$:
\begin{enumerate}
\item  $\{\bfh_{\la,n}|\la\in\Lannr\}$; \item $\{\bfk^\bfj\mid
\bfj\in\mbn^{[-n,n-1]}, \sigma(\bfj)\leq r\}$.
\end{enumerate}
\item For all $-n\leq i\leq n-1$ and $\la\in\Lannr$,
$\bfk_i\bfh_{\la,n}=\up^{\la_i}\bfh_{\la,n}$. \item
$1=\sum_{\la\in\Lannr}\bfh_{\la,n}$.\item The elements $\klan$
satisfy the commuting relations described in \ref{alp} with
$\ttk_\la$ replaced by $\klan$ and $\tte^{(A^+)}$, $\ttf^{(A^-)}$
by $\bfe^{(A^+)}$ and $\bff^{(A^-)}$ for all $A\in\Thnnpm$ and
$\la\in\Lannr$.
\end{enumerate}
\end{Coro}

The elements $\bfh_{\la,n}$ play an important role in a
construction of a spanning set for $\ti \bfV(\iy,r)$.
 Let
 $$\sM_n'=\{\bfe^{(A^+)}\bfh_{\la,n}\bff^{(A^-)}\mid
A\in\Xi^\pm([-n,n-1]),\la\in\Lannr,\la\geq\bfsi(A)\},$$ where
$\la\geq\bfsi(A)$ means $\la_i\geq\sigma_i(A)$ $\forall i$, and
let
$\sM_\iy'=\cup_{n\geq 1}\sM_n'$. 
Recall from \ref{handy} and \eqref{degA} the set
$\La_{n,r,A}$ 
and the function $\deg$.

 \begin{Lem}\label{case1}  For any $n\geq1$, $A\in\Xi^\pm([-n,n-1])$,
 $\la\in\Lannr$,
if $\la\not\in\La_{n,r,A}$, then there exist
$f_{B,\mu}^{A,\la}\in\sZ$ such that
\begin{equation}\label{**}{\bfe}^{(A^+)}\bfh_{\la,n}{\bff}^{(A^-)}=\sum_{B\in\Thnnopm\atop
\mu\in\La_{n,r,A},\deg(B)<\deg(A)}f_{B,\mu}^{A,\la}
\bfe^{(B^+)}\bfh_{\mu,n}\bff^{(B^-)}.\end{equation}
 Hence, $\sM_\iy'$ is a
 spanning set for $\ti \bfV(\iy,r)$.
\end{Lem}
\begin{proof} We first observe, by \ref{property2 of
tilde(ttk)la}(2)\&(3), that all $\bfk_i$ and $\big[ {\bfk_i;c
\atop t} \big]$ are $\sZ$-linear combinations of $\bfh_{\la,n}$
whenever $-n\leq i\leq n-1$. So \eqref{**} implies the assertion
that $\sM_\iy'$ spanns $\ti \bfV(\iy,r)$. We now prove \eqref{**} by applying induction
on $\deg(A)$; see \eqref{degA}.

  Suppose $\la\in\Lannr$ and
$\la_i<\s_i(A)$ where $i$ is minimal. Then $\deg(A)\ge1$.
 If $\deg(A)=1$, then $A^+=E_{i-1,i}$ and
$A^-=0$ or $A^+=0$ and $A^-=E_{i,i-1}$. By \ref{alp}, we have
${\bfe}^{(A^+)}\bfh_{\la,n}{\bff}^{(A^-)}=0$. Assume now
$\deg(A)\ge2$ and \eqref{**} is true for all $A'$ with
$\deg(A')<\deg(A)$.
 Let $A_i=(a_{k,l}^{(i)})$ be the submatrix of $A$
such that $a_{k,l}^{(i)}=a_{k,l}$ if $k,l\leq i$ and
$a_{k,l}^{(i)}=0$ otherwise.
 Since $A\in\Xi^\pm([-n,n-1])$, we can write ${\bfe}^{(A^+)}=\bfm_1{\bfe}^{(A_i^+)}$ and
${\bff}^{(A^-)}=\bff^{(A_i^-)}\bfm'_1$, where $\bfm_1$ is the
monomial of $\bfe_j^{(a)}$ ($-n\leq j<n-1$, $a\geq 0$) and
$\bfm_1'$ is the monomial of $\bff_j^{(a)}$ ($-n\leq j<n-1$,
$a\geq 0$).
 Then
\[{\bfe}^{(A^+)}\klan{\bff}^{(A^-)}=\bfm_1{\bfe}^{(A_i^+)}\klan\bff^{(A_i^-)}\bfm'_1.\]
Since $\la_j\geq \s_j(A_i^+)$ for all $j\leq i$, \ref{property2 of
tilde(ttk)la}(4) implies
\[{\bfe}^{(A^+)}\klan{\bff}^{(A^-)}=
\bfm_1\klanp{\bfe}^{(A_i^+)}\bff^{(A_i^-)}\bfm'_1,\] where
$\la^*\in\Lannr$ with
$\la^*_i=\la-co(A_i^+)-ro(A_i^+)=\la_i-\sum_{j<i}a_{ji}=\la_i-\s_i(A^+)\geq0$.
On the other hand, the commutator formula in \cite[4.1(a)]{Lu89}
(see the remarks right after \cite[2.3]{DP}) implies
\[{\bfe}^{(A_i^+)}\bff^{(A_i^-)}=\bff^{(A_i^-)}{\bfe}^{(A_i^+)}+f,\]
where $f$ is a $\sZ$-linear combination of monomials
$\bfm^\bfe_j\bfh_{j}\bfm^\bff_j$ with
$\deg(\bfm^\bfe_j\bfm^\bff_j)<\deg(A_i)$. Here, $\bfm^\bfe_j$
 is a product of some $\bfe_s^{(a)}$, $\bfm^\bff_j$  is a product of some $\bff_s^{(b)}$, and $\bfh_j$ is a product of
 some $\leb{\ti\bfk_t;c\atop m}\rib$
where\footnote{Our assumption on $A$ guarantees that $\bfe_{n-1}$
and $\bff_{n-1}$ do not appear in ${\bfe}^{(A^+)}$ and
${\bff}^{(A^-)}$. Thus, only the $\ti\bfk_t=\bfk_t\bfk_{t+1}^{-1}$
with $-n\leq t\leq n-2$ occur in the commutator formula; cf.
\ref{definition of S(infty,r)}(g).} $-n\leq s,t\leq n-2$ and
$a,b,c,m\in\mbz$.
 Thus, $\deg(\bfm_1\bfm^\bfe_j\bfm^\bff_j\bfm_1')<\deg(A)$.
Now $\la_i<\s_i(A)$ implies $\la_i^*<\s_i(A^-)=\s_i(A_i^-)$. Hence,
$\klanp\bff^{(A_i^-)}=0$
 by \ref{property2 of tilde(ttk)la}(4).
Since $\bfh_j\in\mathbf B_n^0$ is a $\sZ$-linear combination of
$\bfh_{\mu,n}$ by \ref{property2 of tilde(ttk)la}(1-3), it follows
from \ref{property2 of tilde(ttk)la}(4) that
${\bfe}^{(A^+)}\klan{\bff}^{(A^-)}$ is a $\sZ$-linear combination
of $\bfm_1\bfm_j^{\bfe}\bfh_{\mu,n}\bfm_j^{\bff}\bfm_1'$ with
$\deg(\bfm_1\bfm_j^{\bfe}\bfm_j^{\bff}\bfm_1')<\deg(A)$ and
$\mu\in\Lannr$. Now, since $\Sr$ is a homomorphic image of
$\mathbf U(\infty)$, \ref{Monomial base for U(infty)} implies that
each $\bfm_1\bfm^\bfe_j$ (resp., $\bfm^\bff_j\bfm'_1$) is a
$\sZ$-linear combination of $\bfe^{( B )}$, $B\in\Xi^+([-n,n-1])$
(resp., $\bff^{(B')}$, $B'\in\Xi^-([-n,n-1])$) with
$\deg(B)=\deg(\bfm_1\bfm^\bfe_j)$ (resp.,
$\deg(B')=\deg(\bfm^\bff_j\bfm'_1)$). Thus, each
$\bfm_1\bfm^\bfe_j\bfh_{\mu,n}\bfm^\bff_j\bfm'_1$
 is a $\sZ$-linear combination of $\bfe^{(A')^+}\kmun\bff^{(A')^-}$ with $A'\in\Xi^\pm([-n,n-1])$
 and $\deg(A')<\deg(A)$. Thus, our assertion follows from induction.
\end{proof}

Though each $\sM_n'$ is linearly independent as part of the
integral basis for a $q$-Schur algebra (cf. \eqref{equation1} and
\ref{Ajr basis} below), the set $\sM_\iy'$ is unfortunately not
linearly independent. In fact, we can see that span$(\sM_n')$ is a
subspace of span$(\sM_{n+1}')$, but $\sM_n'$ is not a subset of
$\sM_{n+1}'$. However, we shall use $\sM_\iy'$ to construct
several bases below.

Let $$\tiN_\infty:=\{{\bfe}^{(A^+)}\bfk^{\bfj}{\bff}^{(A^-)}\mid
A\in\Th^{\pm}(\iy),\bfj\in\mbni,\s(\bfj)+\s(A)\leq r\}\subseteq \Sr$$ and, for
$n\geq 1$, let
$$\tiN_n:=\{{\bfe}^{(A^+)}\bfk^{\bfj}{\bff}^{(A^-)}\mid
A\in\Thnnpm,\bfj\in\La_{n,r,A}'\},$$ where $\La_{n,r,A}'$ is
defined in \ref{handy}. Then $\tiN_n\han
 \tiN_{n+1}$ and $\tiN_\infty=\bin_{n\geq 1}\tiN_n$.
 Similarly, we can define $\sN_\infty$ and $\sN_n$ in the algebra $\bfVir$
 so that $\sN_\infty=\pi(\tiN_\infty)$.
 \begin{Coro}\label{lemma for presentation theorem} The set
 $\tiN_\iy$ spans $\Sr$.
\end{Coro}
\begin{proof}
By \ref{case1}, it is enough to prove that, for any $n\geq1$, the
set  $\sM_n'$ lies in the span of $\tiN_n$. In other words, we
need to prove that, for any $A\in\Xi^\pm([-n,n-1])$, and
$\la\in\Lannr$ with $\la\geq\bfsi(A)$, the element
${\bfe}^{(A^+)}\klan{\bff}^{(A^-)}\in\text{span}\,\tiN_n$. We
proceed by induction on $\deg(A)$. If
$\deg(A)=0$, then by \ref{property2 of tilde(ttk)la}(1) we have
${\bfe}^{(A^+)}\klan{\bff}^{(A^-)}=\klan\in\spann\tiN_n$ for
$\la\in\Lannr$. Assume now that $\deg(A)\geq 1$. By \ref{property2
of tilde(ttk)la}(1), we may write, for any $\bfj\in\mbnnn$ with
$j_n=0$,
\begin{equation}\label{equation1}
{\bfe}^{(A^+)}\bfk^{\bfj}{\bff}^{(A^-)}=
\sum_{\mu\in\La_{n,r,A}}\up^{\mu\cdot\bfj}{\bfe}^{(A^+)}\kmun{\bff}^{(A^-)}
+\sum_{\nu\not\in\La_{n,r,A}}\up^{\nu\cdot\bfj}{\bfe}^{(A^+)}\knun{\bff}^{(A^-)}.
\end{equation}
By \eqref{**} and induction, the second sum is in the span of
$\tiN_n$. If $\bfj$ runs over the set $\La_{n,r,A}'$, then the determinant
det$(\up^{\mu\cdot\bfj})_{\mu,\bfj}\neq0$ (see
\ref{handy}). Hence,
${\bfe}^{(A^+)}\kmun{\bff}^{(A^-)}\in\spann\tiN_n$ for all
$\mu\in\La_{n,r,A}$.
 \end{proof}

\begin{Thm}\label{presentation and monomial base for U(infty,r)}
The algebra epimorphism $\pi:\Sr\to\bfV(\iy,r)$ defined in
\eqref{pi} is an isomorphism. In particular, if  a coefficient
$f_{B,\mu}^{A,\la}$ given in \eqref{**} is nonzero, then $B\prec
A$.
\end{Thm}
\begin{proof}
Since $\Sr$ is spanned by $\tiN_\infty$ by \ref{lemma for
presentation theorem}, it is enough to prove that its image
$\pi(\tiN_\infty)=\sN_\infty$ is linearly independent. Fix $n\geq
1$. By \ref{KKK} and \eqref{monomial base for Kr(infty)cc}, for
$A\in\Thnnpm$ and $\bfj\in\mbnnno$, there exist
$f_{B,A}\in\mbq(\up)$ such that
\begin{equation*}
\begin{split}
{\tte}^{(A^+)}\ttk^{\bfj}{\ttf}^{(A^-)}&=\sum_{\la\in\La(\iy,r)
\atop \forall j,\,\la_j\geq\s_j(A)}
\up^{\la\cdot\bfj}{\tte}^{(A^+)}\ttk_{\la}{\ttf}^{(A^-)}+
\sum_{\substack{\mu\in\La(\iy,r) \\
\exists j,\,\mu_j<\s_j(A)}}\up^{\mu\cdot\bfj}
{\tte}^{(A^+)}\ttk_{\mu}{\ttf}^{(A^-)} \\
&=\sum_{\la\in\Lannr \atop \forall
i,\la_i\geq\s_i(A)}\up^{\la\cdot\bfj}{\tte}^{(A^+)}
\ttk_{\la}{\ttf}^{(A^-)}+ \sum_{\la\not\in\Lannr \atop\forall i,
\la_i\geq\s_i(A)}\up^{\la\cdot\bfj}{\tte}^{(A^+)}
\ttk_{\la}{\ttf}^{(A^-)} +\sum_{ B \in\Thir \atop  B \p A}f_{ B
,A}\ttm^{( B )}.
\end{split}
\end{equation*}
Hence, by \eqref{monomial base for Kr(infty)} and \ref{handy} the
set $\sN$ is linearly independent.  Since $\sN_\infty=\bin_{n\geq
1}\mc N_n$ and $\sN_n\han\sN_{n+1}$, it follows that the set
$\sN_\infty$ is linearly independent.

The last assertion follows from  \ref{tri}.
\end{proof}

With the above result  we will identify $\Sr$ with $\bfV(\iy,r)$.
In particular we will identify $\klan$ with $\tth_{\la,n}$,
$\bfe_i$ with $\tte_i$ etc. For $A\in\Th^\pm(\iy)$ and
$\bfj\in\mbni$, let
$$\ttn^{(A,\bfj)}:={\tte}^{(A^+)}\ttk^{\bfj}{\ttf}^{(A^-)}.$$

\begin{Coro}\label{Ajr basis} The set $\sN_\infty=\{\ttn^{(A,\bfj)}\mid A\in\Xi^\pm(\infty),\bfj\in\mbn^\infty,\sigma(\bfj)+\sigma(A)\le r\}$ forms a basis for $\bfV(\iy,r)$.
Moreover, if $\bfj\in\mbn^{[-n,n-1]}$ and
 $\bfj\not\in\La_{n,r,A}'$, then we have, for some
 $p_{B,\bfj'}\in\mbq(v)$,
$${\tte}^{(A^+)}\ttk^{\bfj}{\ttf}^{(A^-)}=\sum_{B\in\Thnnopm\atop B\pr
A, \bfj'\in\La_{n,r,A}'}p_{B,\bfj'}\ttn^{( B ,\bfj')}.$$
\end{Coro}
\begin{proof} Using \eqref{equation1} and \ref{case1}, the last assertion can be
proved as the proof of \cite[6.7]{DFW}.
\end{proof}

The basis $\sN_\infty$ is the counterpart of the basis
\cite[6.6(1)]{DFW} for $q$-Schur algebras. Since the formula
similar to \cite[5.5(c)]{BLM} holds in $\bfVir$, by \ref{Ajr
basis} we get the following result which is the counterpart of the
basis given in \cite[6.6(2)]{DFW}.

\begin{Coro}\label{A(j,r) base of U(infty,r)}
The set  $\mathcal B_\infty=\{A(\bfj,r)\mid
A\in\Th^\pm(\iy),\bfj\in\mbni,\s(\bfj)+\s(A)\leq r\}$  forms a
basis for $\bfV(\iy,r)$.
\end{Coro}

We end this section with another application of the spanning set
$\sM_\iy'$ by displaying an integral monomial basis for the
$\sZ$-form of $\bfV(\iy,r)$; cf. \cite[5.4]{DP}.

Let $V(\iy,r)=\bfxir(U(\iy))$. 
We further put $V^+(\iy,r)=\bfxir(U^+(\iy)),\quad
V^-(\iy,r)=\bfxir(U^-(\iy))\,\,\text{ and
 }\,\,V^0(\iy,r)=\bfxir(U^0(\iy)).$

\begin{Lem}
$(1)$ The set $\{\tte^{(A)}\mid  A\in\Thip,\s(A)\leq r\}\quad
(resp.,\{\ttf^{(A)}\mid  A\in\Thim),\s(A)\leq r\})$ forms a $\sZ$-basis
of $V^+(\iy,r)$ $($resp., $V^-(\iy,r)$.

$(2)$ The set $\sM_\iy^0:=\{\ttk_{\la}\mid\la\in\mbni,\s(\la)\leq r\}$ forms a $\sZ$-basis of $V^0(\iy,r)$.
\end{Lem}
\begin{proof}
Let $A\in\Thip$. If $\s(A)=\sum_i\s_i(A)>r$, then $\tte^{(A)}=
\sum_{\la\in\La(\iy,r)} \tte^{(A)}\ttk_\la=0$ by \ref{KKK} and
\ref{alp}. Now (1) follows from \ref{Monomial base for U(infty)}
and \eqref{monomial base for Kr(infty)}. For $n\geq 1$ we let
$\sM_n^0:=\big\{\prod_{-n\leq i\leq n-1}\leb{\ttk_i;0\atop
\mu_i}\rib\ \big|\ \mu\in\La([-n,n],r)\big\}.$ We fix $n\geq 1$.
For $\la,\mu\in\La([-n,n],r)$ we write $\la\leq_n\mu$ if and only
if $\la_{i}\leq\mu_{i}$ for $-n\leq i\leq n-1$. If $\la_i<\mu_i$
for some $-n\leq i\leq n-1$ then we write $\la<_n\mu$. For
$\la\in\Lannr$ we have by \ref{property2 of tilde(ttk)la}(2)\&(3)
\begin{equation}\label{prod}
\prod_{-n\leq i\leq
n-1}\leb{\ttk_i;0\atop\la_i}\rib=\tth_{\la,n}+\sum_{\mu\in\Lannr,\la<_n\mu}\prod_{-n\leq
i\leq n-1}\leb{\mu_i\atop\la_i}\rib\tth_{\mu,n}. \end{equation}
Hence, the set $\sM_n^0$ is linearly independent and we have
$\spann_{\sZ}\sM_n^0=\spann_{\sZ}\{\tth_{\la,n}|\la\in\Lannr\}$.
Since $\sM_n^0\han\sM_{n+1}^0$ and $\sM_\iy^0=\bin_{n\geq
1}\sM_n^0$, the set $\sM_\iy^0$ is linearly independent. Since
$V^0(\iy,r)=\spann_{\sZ}\{\tth_{\la,n}|\la\in\Lannr, n\geq 1\}$,
we have $V^0(\iy,r)=\spann_{\sZ}\sM_\iy^0$, proving (2).
\end{proof}

The proof above shows that every $\tth_{\la,n}\in V(\infty,r)$.
Hence, the set $\sM_\infty'$ is a subset of the $\sZ$-algebra
$V(\infty,r)$. Thus, \ref{case1} implies immediately the
following.

\begin{Coro}\label{Zspan} The set $\sM_\infty'$ is a spanning set of the $\sZ$-algebra
$V(\infty,r)$.
\end{Coro}

\begin{Prop}
The set
$$\sM_\iy:=\{\tte^{(A^+)}\ttk_{\la}\ttf^{(A^-)}\mid
A\in\Xi^\pm(\iy),\la\in\mbni,\la\geq\bfsi(A),\s(\la)\leq r\}$$
forms a $\sZ$-basis for $V(\iy,r)$.
\end{Prop}
\begin{proof} For any $n\geq 1$, let
$$\sM_n:=\big\{\tte^{(A^+)}\prod_{-n\leq i\leq
n-1}\leb{\ttk_i;0\atop \la_i}\rib\ttf^{(A^-)}\big|
A\in\Xi^\pm([-n,n-1]),\la\in\La_{n,r,A} \big\}.$$ Then,
$\sM_n\han\sM_{n+1}$, $|\sM_n|=|\sM_n'|$ and $\sM_\iy=\bin_{n\geq
1}\sM_n$. Since, for $A\in\Xi^\pm([-n,n-1])$ and
$\la\in\La([-n,n],r)$ with $\la\geq\bfsi(A)$, \eqref{prod} implies
\begin{equation}\label{***}
\tte^{(A^+)}\prod_{-n\leq i\leq n-1}\leb{\ttk_i;0\atop
\la_i}\rib\ttf^{(A^-)}
=\tte^{(A^+)}\tth_{\la,n}\ttf^{(A^-)}+\sum_{\mu\in\Lannr\atop\la<_n\mu}\prod_{-n\leq
i\leq
n-1}\leb{\mu_i\atop\la_i}\rib\tte^{(A^+)}\tth_{\mu,n}\ttf^{(A^-)},
\end{equation}
and $\mu>_n\la$ implies $\mu\geq \bfsi(A)$ (as $\s_i(A)=0$ for
$i\not\in[-n,n-1]$), it follows that all
$\tte^{(A^+)}\tth_{\mu,n}\ttf^{(A^-)}\in\sM_n'$ and
$\spann_{\sZ}\sM_n=\spann_{\sZ}\sM_n'$. Consequently,
$\spann_{\sZ}\sM_\iy=\spann_{\sZ}\sM_\iy'=V(\iy,r)$ by
\ref{Zspan}. On the other hand, the linear independence of
$\sM_n'$ implies that $\sM_n$ is also linearly independent. Hence,
$\sM_\iy$ is linearly independent.
\end{proof}

\begin{Coro}  Let $A\in\Xi^\pm(\iy)$ and $\la\in\mbni$ with $\s(\la)\leq
r$. If $\la_i<\s_i(A)$ for some $i\in\mbz$ then
$$\eap\ttk_\la\faf=\sum_{B\pr
A\atop\mu\geq\bfsi(A),\sigma(\mu)\leq r}
g_{B,\mu}\tte^{(B^+)}\ttk_\mu\ttf^{(B^-)}\quad(g_{B,\mu}\in\sZ,
B\in\Xi^{\pm}(\iy),\mu\in\mbni)$$
\end{Coro}
\begin{proof}
The assertion follows from \eqref{prod}, \eqref{**} and
\eqref{***}.
\end{proof}

\section{Infinite $q$-Schur algebras}

In this section, we will establish the isomorphism between
$\bfU(\infty,r)$ and $\bfVir$ and discuss the relationship between
$\bfU(\infty,r)$ and the infinite $q$-Schur algebra $\bS(\infty,
r)$. We shall mainly work over the field $\mathbb Q(\up)$ though
results like \ref{lemma2 for non-surjective map} and \ref{Cor1 for
non-surjective map}  below continue to hold over $\sZ$.

As in \S 2, let $\fS_r$ be the symmetric groups on $r$ letters.
For $\la\in\La(\iy,r)$, let $\fS_\la$ be the Young subgroup of
$\fS_r$, and let $\frak D_\la$ be the set of  distinguished right
$\frak S_\la$-coset representatives. Then, $\frak D_{\la\mu}=\frak
D_\la\cap\frak D_\mu^{-1}$ is the set of distinguished double
$(\frak S_\la,\fS_\mu)$-coset representatives.

Let $\bfH$ be the
associated Hecke algebras over $\mbq(\up)$ with basis
$\{T_w:=\up^{\ell(w)}\sT_w\}_{w\in\fS_r}$, and let
$$x_\la=\sum_{w\in\frak S_\la}T_w.$$
Recall from \S2 the $\bfH$-module $\bfOg_{\eta}^{\ot r}$. Clearly,
there is an $\bfH$-module embedding $i_n:\bfOg_{[-n,n]}^{\ot r}\to
\bfOg_{[-n-1,n+1]}^{\ot r}$, and $\{(\bfOg_{[-n,n]}^{\ot
r},i_n)\}_{n\geq1}$ forms a direct system.

\begin{Lem}\label{lemma2 for non-surjective map} The $\bfH$-module
$\bfOg_{\iy}^{\ot r}$ is the
 direct limit of $\bfH$-modules $\{(\bfOg_{[-n,n]}^{\ot
 r},i_n)\}_{n\geq1}$.
Thus, $\bfH$-module $\bfOg_\iy^{\ot r}$ is isomorphic to the
$\bfH$-module $\oplus_{\la\in\La(\iy,r)}x_{\la}\bfH$.
\end{Lem}
\begin{proof}By \cite[5.1]{DPW}, we know there is a $\bfH$-module isomorphism
 $\al_n:\bfOg_{[-n,n]}^{\ot r}\ra\oplus_{\la\in\Lannr}x_\la\bfH$ and the
 following diagram are commutative.
$$\begin{CD}
\bfOg_{[-n,n]}^{\ot r} @>\al_n>>\bigoplus\limits_{\la\in\Lannr}x_{\la}\bfH\\
@V i_n VV  @VV i_{n}' V\\
\bfOg_{[-n-1,n+1]}^{\ot r} @>>\al_{n+1}>
\bigoplus\limits_{\la\in\La([-n-1,n+1],r)}x_{\la}\bfH,
\end{CD}$$
where $i_n$ and $i_n'$ are natural injections. Since
$\bfOg_{\iy}^{\ot r}$ is the direct limit of
$\{\bfOg_{[-n,n]}^{\ot r}\}_{n\geq1}$ and
$\oplus_{\la\in\La(\iy,r)}x_\la\bfH$ is the direct limit of
$\{\oplus_{\la\in\La([-n,n],r)}x_\la\bfH\}_{n\geq1}$, the result
follows.
\end{proof}
\begin{Prop}\label{Cor1 for non-surjective map}
We have the following algebra isomorphisms
\begin{equation*}
\begin{split}
\bfKir&\overset\theta\cong \bop_{\mu\in\Lair}\bigoplus_{\la\in\La(\infty,r)}\Hom_{\bfH}(x_{\mu}\bfH,x_{\la}\bfH)\\
\iSr&\overset{\tilde\theta}\cong\prod_{\mu\in\La(\iy,r)}\bigoplus_{\la\in\La(\iy,r)}\Hom_ {\bfH}(x_\mu\bfH,x_\la\bfH).\\
\end{split}
\end{equation*}
In particular, the algebra $\bfKir$ can be regarded as a
subalgebra (without the identity) of $\iSr$.
\end{Prop}
\begin{proof} It is known (see, e.g., \cite{Du95}) that, for any $n
\geq0$, the $q$-Schur algebras $\nKrb$ is isomorphic to
$\bigoplus\limits_{\la,\mu\in\Lannr}
 \Hom_{\bfH}(x_{\la}\bfH,x_{\mu}\bfH)$.
Since $\bfKir$ is the direct limit of $\{\nKrb\}_{n\geq1}$, the
first assertion follows. By \ref{lemma2 for non-surjective map},
we have
$$\iSr\cong\prod_{\mu\in\La(\iy,r)}\Hom_{\bfH}(x_\mu\bfH,\bigoplus_{\la\in\La(\iy,r)}x_\la\bfH).$$
Now, one checks easily that, for each $\mu\in\La(\iy,r)$, there is
a natural isomorphism
$$\bigoplus_{\la\in\La(\iy,r)}\Hom_{\bfH}(x_\mu\bfH,x_\la\bfH)\cong\Hom_\bfH(x_\mu\bfH,\bigoplus_{\la\in\La(\iy,r)}x_\la\bfH),$$
which induces the required algebra isomorphism.
\end{proof}

\newcommand{\tla}{t^\lambda}

The isomorphism $\theta$ can be made explicit.  For
$\la\in\La(\iy,r)$, let $Y(\la)$ be the Young diagram of $\la$
which is a collection of boxes, arranged in left justified rows
with $\la_i$ boxes in row $i$ for all $i\in\mbz$, and let $\tla$
be the $\la$-tableau in which the numbers $1,2\cdots,r$ appear in
order from left to right down successive (non-empty) rows of
$Y(\la)$. Let $R^\la_i$ ($i\in\mbz$) be the set of entries in the
$i$-th row of $\tla$. Clearly, $R^\la_i\neq\emptyset$ if and only
if $\la_i>0$.
  Let
\begin{equation}{\frak D}(\infty,r)=\{(\la,d,\mu)\mid
\la,\mu\in\La(\iy,r),\ d\in\frak D_{\la\mu}\}.\end{equation}
 By \cite[(1.3.10)]{JK}, every element $(\la,d,\mu)\in{\frak
D}(\infty,r)$ defines a matrix $A=(a_{i,j})_{i,j\in\mbz}\in\Thir$
such that $a_{i,j}=|R^\la_i\cap dR^\mu_j|$. This defines a
bijective map
$$\jmath:{\frak D}(\infty,r)\lra\Xi(\infty,r).$$
Define, for any $\la,\mu\in\La(\infty,r)$ and $w\in\frak
D_{\la,\mu}$, a map
\begin{equation}\label{phi basis}\phi_{\la\mu}^w:\oplus_{\la\in\La(\infty,r)}x_{\la}\mc H\to
\oplus_{\la\in\La(\infty,r)}x_{\la}\mc H\end{equation} by setting
$\phi_{\la\mu}^w(x_\nu h)=\delta_{\mu,\nu}\sum_{x\in\fS_\la
w\fS_\mu}T_xh.$ Then, the set $\{\phi_{\la\mu}^w\mid (\la,d,\mu)\in
{\frak D}(\infty,r)\}$ forms a basis for
$\bop_{\mu\in\Lair}\bigoplus_{\la\in\La(\infty,r)}\Hom_{\bfH}(x_{\mu}\bfH,x_{\la}\bfH).$
 By \cite[1.4]{Du95}, the
isomorphism $\theta$ is induced by $\jmath$. In other words, if
$\jmath(\la,w,\mu)=A$, then $\theta(e_A)=\phi_{\la\mu}^w$. In
particular, by \ref{KKK}, we have
$\theta(\ttk_\la)=\phi_{\la\la}^1$.

We now identify $\iSr$ as
$\prod_{\mu\in\La(\iy,r)}\bigoplus_{\la\in\La(\iy,r)}\Hom_
{\bfH}(x_\mu\bfH,x_\la\bfH)$ under $\tilde\theta$. Thus, the
elements of $\iSr$ have the form
$\biggl(\sum_{\la\in\Lair}f_{\la,\mu}\biggr)_{\mu\in\Lair}$, where
$f_{\la,\mu}\in\Hom_{\bfH}(x_{\mu}\bfH,x_{\la}\bfH)$ for all
$\la,\mu\in\Lair$, which represents the map given by
 $$\biggl(\sum_{\la\in\Lair}f_{\la,\mu}\biggr)_{\mu\in\Lair}\bigl(\sum_{\nu\in\Lair}a_{\nu}\bigr)=\sum_{\la,\mu\in\Lair}f_{\la,\mu}(a_{\mu}),$$
where $\sum_\nu a_{\nu}\in \oplus_\la x_{\la}\bfH$. This
identification allows us to regard the infinite $q$-Schur algebra
$\iSr$ as the $\dagger$-completion of $\bfKir$ (see \ref{completion}).

\begin{Prop}\label{iSr as dagger} Let $\tibfKir$ be the completion algebra of $\bfKir$
constructed as in \ref{completion}. Then there is an algebra
isomorphism \begin{equation*}
\barbfzetar:\tibfKir\st{\thicksim}{\lra}\iSr,
\,\, \sum_{A\in\Thir}\bt_{A}[A]\lm \biggl(\sum_{co(A)=\mu\atop
A\in\Thir}\bt_A[A]\biggr)_{\mu\in\La(\iy,r)}.
\end{equation*}
\end{Prop}
\begin{proof} Clearly, $\barbfzetar$ is induced by
the natural monomorphism $\barbfzetar:\bfKir\ra\iSr$ given in
\ref{Cor1 for non-surjective map}.
\end{proof}

 By \ref{completion}, there is another completion algebra $\hbfKir$
 which is a subalgebra of $\tibfKir$ and which contains
 $\iVr$ as a subalgebras. Thus, restriction gives an algebra monomorphism $\varsigma_r:\iVr\to\iSr$.

We are now ready to establish the isomorphism $\iUr\cong\iVr$.

There is a similar basis
 $\big\{\phi_{\la\mu}^d\mid \la,\mu\in\Lannr,\ d\in
 \frak D_{\la\mu}\big\}$ for $\bfS([-n,n],r)$ by
 \cite{DJ91}, and $[A]=\up^{-d_A}\phi_{\la\mu}^d$, where
 $A=\jmath(\la,d,\mu)$,  if we identify $\bfUnr$ with
 $\bfS([-n,n],r)$ (cf. \cite[A.1]{Du92}).

 \begin{Thm}\label{not onto} The following diagram is commutative
 $$\begin{CD}
\bfV(\infty) @>\xi_r>>\iVr\\
@V \sim VV  @VV \varsigma_r V\\
\iU @>>\zeta_r> \iSr
\end{CD}$$
 Hence, we have an isomorphism $\iUr\cong\iVr$.
 Moreover, the map $\bfzetar:\iU\ra\iSr$ is not surjective for any $r\geq 1$.
 \end{Thm}
 \begin{proof}  Recall from \ref{map zr} that we may identify $\bfU(\infty)$ with $\bfV(\infty)$.
 Fix some $i\in\mbz$.
 For $\og\in\bfOg_\iy^{\ot r}$, we choose $n\geq 1$ such that $-n\leq  i<n$ and
 $\og\in\bfOg_{[-n,n]}^{\ot r}$. Then $E_i\cdot\og=E_{i,i+1}(\bfl,r)_{[-n,n]}\cdot\og$ by a result for $q$-Schur algebras.
 It is clear that we have  $E_{i,i+1}(\bfl,r)_{[-n,n]}\cdot\og=E_{i,i+1}(\bfl,r)_\iy\cdot\og$.
 Hence,  $E_i\cdot\og=E_{i,i+1}(\bfl,r)_\iy\cdot\og$ for any $\og\in\bfOg_\iy^{\ot r}$.
 Similarly, we have $F_i\cdot\og=E_{i+1,i}(\bfl,r)_\iy\cdot\og$ and $K_i\cdot\og=0(\be_i,r)_\iy\cdot\og$ for any $\og\in\bfOg_\iy^{\ot r}$.
 The first assertion follows.

To see the last statement, it suffices to prove that the injective
map ${\varsigma}_r:\widehat\bfK(\infty,r)\ra\iSr$ is not
surjective. Fix a $\la\in\La(\infty,r)$. For any
$\mu\in\La(\infty,r)$, construct a matrix $A_\mu\in\Xi(\infty,r)$
such that $ro(A_\mu)=\la$ and $co(A_\mu)=\mu$. Clearly, the
element $([A_\mu])_{\mu\in\La(\infty,r)}$ belongs to $\iSr$, but
not in $\widehat\bfK(\infty,r)$.
 \end{proof}

 \begin{Rem} Recall from \cite{Ji} that, in the finite case, the
 $(\bfU(n),\bH)$-bimodule structure on $\bfOg_n^{\ot r}$ induces
 two algebra epimorphisms
 $$\bfU(n)\twoheadrightarrow\End_\bfH(\bfOg_n^{\ot r}),\qquad
 \bfH\twoheadrightarrow\End_{\bfU(n)}(\bfOg_n^{\ot r})$$
 (see \cite{DPS3} for the roots-of-unity case). This is the so-called {\it Schur--Weyl duality}.
 The theorem above shows that this duality is no longer true in
 the infinite case. However, the second epimorphism continues to hold;
 see \ref{bimodule isomorphism} below.
 \end{Rem}

 With the above result, we will identify $\iUr$ with $\iVr$. Thus,
 the algebra homomorphism $\bfzetar:\iU\ra\iSr$ sends $E_i,K_i,F_i$
 to $\tte_i,\ttk_i,\ttf_i$, respectively. Now,
by applying $\bfzetar$ to the graded components in \eqref{alg grading}, 
\ref{alp} implies immediately the following.

 \begin{Coro}\label{lemma2 for the map dzr}
Let $\la\in\Lair$ and $\nu',\nu''\in\mbz\iPi$. If
$t'\in\iU_{\nu'}$ and $t''\in\iU_{\nu''}$, then
$$\zr(t')[\diag(\la)]=\begin{cases}[\diag(\la+\nu')]\zr(t') &
\mathrm{if\ } \la+\nu'\in\La(\iy,r);\\
0&\mathrm{otherwise},
\end{cases}$$ and
$$[\diag(\la)]\zr(t'')=\begin{cases}\zr(t'')[\diag(\la-\nu'')]& \mathrm{if\ }\la-\nu''\in\La(\iy,r);\\
0&\mathrm{otherwise}.\end{cases}$$
\end{Coro}

\section{Modified quantum $\frak{gl}_{\iy}$ and related algebras}

In this section, we identify $\mathbf U(\infty)$ with $\mathbf V(\infty)$ and $\mathbf U(\infty,r)$ with $\bfVir$.
 Thus, $E_h=E_{h,h+1}(\mathbf0)$ and
$F_h=E_{h+1,h}(\mathbf0)$ for all $h\in\mbz$.

The interpretation of the infinite $q$-Schur algebra $\iSr$ as the
$\dagger$-completion $\h\bfK^\dagger(\iy,r)$ of $\bfKir$ suggests
us to introduce the $\dagger$-completion $\h\bfK^\dagger(\iy)$ of
the algebra $\bfKi$ which contains $\h\bfK(\iy)$, and hence $\iU$,
as a subalgebra. Thus, it is natural to expect that the map
$\bfzetar:\iU\ra\iSr$ given in \eqref{bfzeta1} may be extended to
an epimorphism from $\h\bfK^\dagger(\iy)$ to $\iSr$. In this
section, we will establish this epimorphism through an
investigation of the epimorphism $\dot\bfzetar:\bfKi\to\bfKir$
induced by $\bfzetar$ via the modified quantum group $\diU$ of
$\iU$ (see \cite[23.1]{Lu93}).

Recall from \ref{map zr} that $\iU$ has a basis
$\{A(\bfj)\}_{A,\bfj}$. Thus, for any $\la,\mu\in\mbzi$, there is
a linear map from $\iU$ to $\bfKi$ sending $u$ to
$[\diag(\la)]u[\diag(\mu)]$. Let  $${}_\la\mathbf
K_\mu=\sum_{\bfj\in\mbzi}(K^\bfj-
 \up^{\la\cdot\bfj})\iU+\sum_{\bfj\in\mbzi}\iU(K^\bfj-\up^{\mu\cdot\bfj})\,\,\text{ and }\,\,_\la\overline{\iU}_\mu:=\iU/{}_\la\mathbf
K_\mu.$$ Since $[\diag(\la)]{}_\la\mathbf K_\mu[\diag(\mu)]=0$,
this map induces a linear map $_\la\overline{\iU}_\mu\to\bfKi$
sending $\pi_{\la\mu}(u)$ to $[\diag(\la)]u[\diag(\mu)]$, where
$\pi_{\la\mu}:\iU\ra{_\la}\overline{\iU}_\mu$ is the  canonical
projection. Thus, we obtain a linear map
$$f:\diU:=\bop\limits_{\la,\mu\in\mbzi}{_\la}\overline{\iU}_\mu\lra\bfKi.$$
We will introduce a multiplication in $\diU$ and prove that $f$ is
an algebra isomorphism. We need some preparation.

 \begin{Lem}\label{lemma1 for isomorphism between dot U and K(infty)}
Let $A\in\Th^\pm(\iy)$ and $\la,\mu\in\mbzi$. If
$\la-\mu\not=ro(A)-co(A)$, then we have $\pi_{\la\mu}(A(\bfj))=0$
for any $\bfj\in\mbzi$.
\end{Lem}
\begin{proof}Since
$\mu-co(A)\not=\la-ro(A)$, there exist $\bfj'\in\mbzi$ such that
$\bfj'\cdot(\mu-co(A))\not=\bfj'\cdot(\la-ro(A))$. Since
$\up^{-\bfj'\cdot
ro(A)}0(\bfj')A(\bfj)=A(\bfj+\bfj')=\up^{-\bfj'\cdot
co(A)}A(\bfj)0(\bfj')$
 we have
for any $\bfj\in\mbzi$
$$(\up^{\bfj'\cdot(\mu-co(A))}-\up^{\bfj'\cdot(\la-ro(A))})A(\bfj)=\up^{-\bfj'\cdot ro(A)}
(K^{\bfj'}-\up^{\la\cdot\bfj'})A(\bfj) -\up^{-\bfj'\cdot
co(A)}A(\bfj)(K^{\bfj'}-\up^{\mu\cdot\bfj'}).$$ The assertion
follows.
\end{proof}

 Recall the algebra grading of $\iU$
 given in \eqref{alg grading}.

\begin{Lem}\label{lemma4 for isomorphism between dot U and K(infty)}
For $\la,\mu\in\mbzi$ and $\nu\in\mbz\iPi$. If $\nu\not=\la-\mu$
then we have $\pi_{\la\mu}(\iU_\nu )=0$. Hence,
${_\la}\overline{\iU}_\mu=\pi_{\la\mu}(\iU_{\la-\mu})$.
\end{Lem}
\begin{proof}
Let $t\in\iU_\nu $. Since $K^\bfj t=\up^{\nu\cdot\bfj}tK^\bfj$ for
$\bfj\in\mbzi$, we have
$\up^{\la\cdot\bfj}\pi_{\la,\mu}(t)=\pi_{\la\mu}(K^\bfj
t)=\pi_{\la\mu}
(\up^{\nu\cdot\bfj}tK^\bfj)=\up^{\mu\cdot\bfj+\nu\cdot\bfj}\pi_{\la\mu}(t)$
for all $\bfj\in\mbzi$. Since $\nu\not=\la-\mu$ there exist
$\bfj\in\mbzi$ such that $\nu\cdot\bfj\not=(\la-\mu)\cdot\bfj$.
Hence, $\pi_{\la\mu}(t)=0$.
\end{proof}

For any $\la',\mu',\la'',\mu''\in\mbzi$ with
$\la'-\mu',\la''-\mu''\in\mbz\iPi$ and any $t\in\iU_{\la'-\mu'}$,
$s\in\iU_{\la''-\mu''}$, define the product in $\diU$
$$\pi_{\la'\mu'}(t)\pi_{\la''\mu''}(s)=\begin{cases}\pi_{\la'\mu''}(ts),
& \text{if } \mu'=\la''\\
0& \text{otherwise}.
\end{cases}$$
Using \ref{lemma4 for isomorphism between dot U and K(infty)} one
can easily check the above product defines  an associative
$\mbq(\up)$-algebra structure on $\diU$.

\begin{Thm}\label{isomorphism between dot U and K(infty)} The
linear map $f:\diU\ra\bfKi$ sending $\pi_{\la\mu}(u)$ to
$[\diag(\la)]u[\diag(\mu)]$ for all $u\in\iU$ and
$\la,\mu\in\mbzi$, is an algebra isomorphism.
\end{Thm}
\begin{proof} Since, for any $A\in\ti\Th^\pm(\iy)$, $\bfj\in\mbz^\iy$ and $\lambda,\mu\in\mbz^\infty$ with $\lambda-\mu=ro(A)-co(A)$, we have
\begin{equation}\label{andy}
[\diag(\la)]A({\bfj})[\diag(\mu)]=\up^{(\la-ro(A))\cdot\bfj}[A+\diag(\la-ro(A))]=\up^{(\mu-co(A))\cdot\bfj}[A+\diag(\mu-co(A))],
\end{equation}
it follows that $[A]=[\diag(ro(A))]A^\pm({\bf0})[\diag(co(A))]$
for all $A\in\ti\Xi(\iy)$, and so
$$\bfKi=\bop_{\la,\mu\in\mbzi}[\diag(\la)]\iU[\diag(\mu)].$$ So $f$ is surjective.
Suppose now $\sum_{\la,\mu}\pi_{\la\mu}(u)\in\ker(f)$, where
$u=\sum_{A\in\Th^\pm(\iy)\atop\bfj\in\mbzi}\beta_{_{A,\bfj}}A(\bfj)$
with $\beta_{_{A,\bfj}}\in\mbq(\up)$. Then,
$[\diag(\la)]u[\diag(\mu)]=0$ for all $\la,\mu$. Let
$$u_{\la\mu}=\sum_{A\in\Th^\pm(\iy),\bfj\in\mbzi\atop\la-\mu=ro(A)-co(A)}\beta_{_{A,\bfj}}A(\bfj).$$
 By \ref{lemma1 for isomorphism between
dot U and K(infty)} we have
$\pi_{\la\mu}(u)=\pi_{\la\mu}(u_{\la\mu})$. On the other hand,
\eqref{andy} implies
$$0=[\diag(\la)]u_{\la\mu}[\diag(\mu)]=\sum_{A\in\Th^\pm(\iy)\atop\la-\mu=ro(A)-co(A)}
\left(\sum_{\bfj\in\mbzi}\up^{(\la-ro(A))\cdot
\bfj}\beta_{_{A,\bfj}}\right)[A+\diag(\la)-ro(A)].$$ Hence,
$\sum_{\bfj\in\mbzi}\up^{(\la-ro(A))\cdot\bfj}\beta_{_{A,\bfj}}=0$
for any $A\in\Th^\pm(\iy)$ with $\la-\mu=ro(A)-co(A)$. Since
$A(\bfj)=\up^{-\bfj\cdot ro(A)}0(\bfj)A(\bfl)$ we have
\begin{equation*}\begin{split}u_{\la\mu}&=\sum_{A\in\Th^\pm(\iy),\bfj\in\mbzi\atop\la-\mu=ro(A)-co(A)}
\beta_{_{A,\bfj}}A(\bfj)-\sum_{A\in\Th^\pm(\iy),\bfj\in\mbzi\atop\la-\mu=ro(A)-co(A)}
\up^{(\la-ro(A))\cdot\bfj}\beta_{_{A,\bfj}}A(0)\\
&=\sum_{A\in\Th^\pm(\iy),\bfj\in\mbzi\atop\la-\mu=ro(A)-co(A)}(K^\bfj-\up^{\la\cdot\bfj})
\up^{-\bfj\cdot ro(A)}\beta_{_{A,\bfj}}A(0)\in{}_\la\mathbf
K_\mu.\end{split}\end{equation*} Hence,
$\pi_{\la\mu}(u)=\pi_{\la\mu}(u_{\la\mu})=0$ for all $\la,\mu$. So
$f$ is injective.

 We now prove $f$ is an algebra homomorphism. Let $u_1,u_2\in\iU$
 and $\la',\mu',\la'',\mu''\in\mbzi$. If $\mu'\not=\la''$ or
$\la'-\mu'\not\in\mbz\Pi(\iy)$ or $\la''-\mu''\not\in\mbz\Pi(\iy)$
 then by definition and \ref{lemma4 for isomorphism between dot U and K(infty)}
$$f(\pi_{\la'\mu'}(u_1)\pi_{\la''\mu''}(u_2))=0=f(\pi_{\la'\mu'}(u_1))f(\pi_{\la''\mu''}(u_2)).$$
It remains to prove the case when $\mu'=\la''$,
$\la'-\mu'\in\mbz\Pi(\iy)$ and $\la''-\mu''\in\mbz\Pi(\iy)$. By
\ref{lemma4 for isomorphism between dot U and K(infty)}, we may
assume $u_1\in\iU_{\la'-\mu'}$ and $u_2\in\iU_{\la''-\mu''}$.
Observe that, for $u=E^{(A^+)} K^\bfj F^{(A^-)}\in\iU_{\la-\mu}$,
where $\la,\mu\in\mbzi$ with $\la-\mu\in\mbz\Pi(\iy)$, since
$\la-\mu=\sum_{i\leq h<j}(a_{ij}-a_{ji})\al_h$, it follows from
\ref{the commute formula in the completion algebra of K(infty)}
that
$$[\diag(\la)]u[\diag(\mu)]=[\diag(\la)][\diag(\mu+\sum_{i\leq
k<j}(a_{ij}-a_{ji})\al_k)]u=[\diag(\la)]u.$$ Hence, by
\ref{Monomial base for U(infty)},
$[\diag(\la)]u'[\diag(\mu)]=[\diag(\la)]u'$ for any
$u'\in\iU_{\la-\mu}$. Thus, if $\mu'=\la''$, then
\begin{equation*}
\begin{split}
f(\pi_{\la',\mu'}(u_1)\pi_{\mu',\mu''}(u_2))&=f(\pi_{\la'\mu''}(u_1u_2))=[\diag(\la')]u_1u_2[\diag(\mu'')]\\
&=([\diag(\la')]u_1[\diag(\mu')])u_2[\diag(\mu'')]\\&=f(\pi_{\la',\mu'}(u_1))f(\pi_{\mu',\mu''}(u_2)).
\end{split}
\end{equation*}
as required.
\end{proof}
With the above result, we shall identify $\diU$ with $\bfKi$. In
particular, we shall identify $\pi_{\la\mu}(u)$ with
$[\diag(\la)]u[\diag(\mu)]$ for $u\in\iU$ and $\la,\mu\in\mbzi$.

Since both algebras $\bfKi=\diU$ and $\bfV(\iy)=\iU$ are
subalgebras of $\h\bfK(\iy)$, the definition of $\bfV(\iy)$
implies that the algebra $\bfKi$ is a $\iU$-bimodule with an
action induced by multiplication. We now have the following
interpretation of this bimodule structure.

\begin{Coro}\label{bimodule}  The algebra $\iU$ acts on $\diU$ by the following rules:
$$t'\pi_{\la\mu}(s)t''=\pi_{\la+\nu',\mu-\nu''}(t'st'')$$
for all $t'\in\iU(\nu')$, $t''\in\iU(\nu'')$ and $s\in\iU$, where
$\la,\mu\in\mbzi$ and $\nu',\nu''\in\mbz\iPi$.
\end{Coro}
\begin{proof}
By \ref{the commute formula in the completion algebra of K(infty)}
and \ref{Monomial base for U(infty)} we have
$t[\diag(\la)]=[\diag(\la+\nu)]t$ and
$[\diag(\la)]t=t[\diag(\la-\nu)]$ for $t\in\iU_\nu $. Hence,
 we have $t'\pi_{\la\mu}(s)t''=t'[\diag(\la)]s[\diag(\mu)]t''
=[\diag(\la+\nu')]t'st''[\diag(\mu+\nu'')]=\pi_{\la+\nu',\mu-\nu''}(t'st'')$.
\end{proof}

 Similar to \cite[3.4]{Lu00}, we  define a
map $\dzr$ from $\diU$ to $\iUr$ as follows.
\begin{Thm}
The map $\dzr:\diU\ra\iUr$ defined by
$$\dzr(\pi_{\la\mu}(u))=\begin{cases}[\diag(\la)]\zr(u)[\diag(\mu)]&\mathrm{if\ }\la,\mu\in\La(\iy,r);\\
0& \mathrm{otherwise}
\end{cases}$$
for $u\in\iU$ and $\la,\mu\in\mbzi$  is an algebra homomorphism.
\end{Thm}
\begin{proof}
We first prove that $\dzr$ is well-defined. Assume
$\pi_{\la\mu}(u)=0$ where $u\in\iU$ and $\la,\mu\in\La(\iy,r)$. By
the definition of $_\la\iU_\mu$ we can write
$$u=\sum_{\bfj\in\mbzi}(K^\bfj-\up^{\la\cdot\bfj})u_{\bfj}'+
\sum_{\bfj\in\mbzi}u_{\bfj}''(K^\bfj-\up^{\mu\cdot\bfj})
$$ where $u_{\bfj}',u_{\bfj}''\in\iU$. By \ref{KKK} we have
 $[\diag(\la)](\ttk^{\bfj}-\up^{\la\cdot\bfj})=0=(\ttk^{\bfj}-\up^{\mu\cdot\bfj})[\diag(\mu)]$
 for $\bfj\in\mbzi$. It follows that $[\diag(\la)]\zr(u)[\diag(\mu)]=0$.
 Hence, $\dzr$ is well defined.

 For convenience, we let $[\diag(\la)]=0\in\iUr$ for
$\la\not\in\La(\iy,r)$.
Let $t\in\iU(\la'-\mu')$ and $s\in\iU(\la''-\mu'')$ where $\la',\mu',\la'',\mu''\in\mbzi$ such that
$\la'-\mu',\la''-\mu''\in\mbz\iPi$. If $\mu'\not=\la''$ then by
the definition of $\dzr$ we have
$\dzr(\pi_{\la'\mu'}(t)\pi_{\la''\mu''}(s))=0=\dzr(\pi_{\la'\mu'}(t))\dzr(\pi_{\la''\mu''}(s))$.
Now we assume $\mu'=\la''$. Then by \ref{lemma2 for the map dzr}
we have
$\dzr(\pi_{\la'\mu'}(t))\dzr(\pi_{\mu'\mu''}(s))=[\diag(\la')]\zr(t)[\diag(\mu')]\zr(s)[\diag(\mu'')]
=[\diag(\la')]\zr(t)\zr(s)[\diag(\mu'')]=[\diag(\la')]\zr(ts)[\diag(\mu'')]=
\dzr(\pi_{\la'\mu''}(ts))
=\dzr(\pi_{\la'\mu'}(t)\pi_{\mu'\mu''}(s))$. The result follows.
\end{proof}

\begin{Prop}\label{lemma4 for the map dzr}
The map $\dzr$ satisfies the following property$:$
$$\dzr(u_1u_2u_3)=\zr(u_1)\dzr(u_2)\zr(u_3)$$
where $u_1,u_3\in\iU$ and $u_2\in\diU$.
\end{Prop}
\begin{proof}We assume $u'\in\iU(\nu')$, $u''\in\iU(\nu'')$ and $u\in\iU(\la-\mu)$
where $\nu',\nu'',\la-\mu\in\mbz\iPi$ and $\la,\mu\in\mbzi$. By
\ref{bimodule} we have
$$\dzr(u'\pi_{\la\mu}(u)u'')=\dzr(\pi_{\la+\nu',\mu-\nu''}(u'uu''))=\begin{cases}
\zr(u')\zr(u)\zr(u'')[\diag(\mu-\nu'')]& \text{if
}\mu-\nu''\in\La(\iy,r); \\
0& \text{otherwise}.\end{cases}$$ Hence, by \ref{lemma2 for the map
dzr} we have
$\zr(u')\dzr(\pi_{\la\mu}(u))\zr(u'')=\zr(u')\zr(u)[\diag(\mu)]\zr(u'')=\dzr(u'\pi_{\la\mu}(u)u'')$.
\end{proof}
\begin{Prop}\label{image of [A] under the map dzr}
Let $A\in\ti\Th(\iy)$. Then we have
$$\dzr([A])=\begin{cases}[A]& \mathrm{if\ }A\in\Thir;\\
0&  \mathrm{otherwise}.\end{cases}$$
In particular we have $\dzr(\bfK(\iy))=\dzr(\diU)=\bfKir$.
\end{Prop}
\begin{proof}Let $\la=ro(A)$ and $\mu=co(A)$. If either
$\la\not\in\La(\iy,r)$ or $\mu\not\in\La(\iy,r)$, then we have
$\dzr([A])=\dzr([\diag(\la)]A^\pm(\bfl)[\diag(\mu)])=0$. Now we
assume $\la,\mu\in\La(\iy,r)$. Then we have
\begin{equation*}\begin{split}\dzr([A])&=\dzr([\diag(\la)]A^\pm(\bfl)[\diag(\mu)])\\
&=[\diag(\la)]\zr(A^\pm(\bfl))[\diag(\mu)]\\
&=[\diag(\la)]A^\pm(\bfl,r)[\diag(\mu)]\\
&=\begin{cases}[A]& \mathrm{if\ }A\in\Thir;\\
0&  \mathrm{otherwise}.\end{cases}
\end{split}
\end{equation*}
The result follows.
\end{proof}
\begin{Rem}
The natural linear map from $\bfKir$ to $\bfK(\iy)$ by sending
$[A]$ to $[A]$ for $A\in\Thir$ is not an algebra homomorphism. For
example, in the algebra $\bfKir$ we have
$[E_{1,2}]\cdot[E_{2,1}]=[E_{1,1}]$.
However we have $[E_{1,2}]\cdot[E_{2,1}]=[E_{1,1}]+[A]$, where
$A=\left(\begin{smallmatrix}0&1
\\1&-1\end{smallmatrix}\right)$
in the algebra $\bfK(\iy)$.
\end{Rem}


By \ref{completion} we may construct the completion algebra $\iKt$ of $\bfKi$ such that $\iKh$ becomes an subalgebra of $\iKt$.

\begin{Thm}There is an  algebra epimorphism $\tzr$ from $\iKt$ to
$\iSr$ defined  by sending $\sum_{A\in\ti\Th(\iy)}\bt_{A}[A]$ to
$\sum_{A\in\Thir}\bt_{A}[A]$. Moreover we have
$\tzr(\hbfKi)=\hbfKir$ and $\tzr|_{\iU}=\zr$.
\end{Thm}
\begin{proof}By \ref{image of [A] under the map dzr} there is an  algebra homomorphism $\tzr$ from $\iKt$ to
$\iSr$ by sending
$\sum_{A\in\ti\Th(\iy)}\bt_{A}[A]$ to
$\sum_{A\in\ti\Th(\iy)}\bt_{A}\dzr([A])=\sum_{A\in\Thir}\bt_{A}[A]$. It is clear we have $\tzr(\hbfKi)=\hbfKir$.
By \ref{Cor1 for non-surjective map} the map $\tzr$ is surjective.
By  \ref{map zr} and \ref{image of [A] under the map
dzr} we have $\zr(A(\bfj))=A(\bfj,r)=\tzr(A(\bfj))$ for
$A\in\Th^\pm(\iy)$ and $\bfj\in\mbzi$. Hence, we have
$\tzr|_{\iU}=\zr$.
\end{proof}
\begin{Rem}
By \cite[3.10]{BLM} there is an antiautomorphism $\tau_r$ on
$\bfKir$ defined by $\tau_r([A])=[A^T]$ for $A\in\Thir$. This
induces, by \ref{stabilization property}, an antiautomorphism
$\tau$ on $\bfKi$ defined by $\tau([A])=[A^T]$ for $A\in\tiThi$.
Hence, 
we have
$\tibfKil\cong(\tibfKi)^{\text{op}}$ and $\tibfKirl\cong(\tibfKir)^{\text{op}}$.
\end{Rem}

\section{Highest weight representations of $\iU$}

In the rest of the paper, we will discuss the representation
theory of $\iU$. There
are two types of representations for $\iU$. The first type
consists of highest weight representations, while the second type
consists of ``polynomial representations'' arising from
representations of infinite $q$-Schur algebras. In the quantum $\frak{gl}_n$ case, it is well-known
that polynomial representations more or less determine all highest weight representations.
However, we shall see in Theorem \ref{classfication for simple module in tilde Cr} that
polynomial representations and highest weight representations of $\iU$ are almost mutually exclusive.
In this section,
we first discuss the highest weight representations of $\iU$,
following the approach used in \cite{Lu93}.

 For a consecutive segment $\eta$ of $\mathbb Z$, let
\begin{equation}\label{wt poset}
\aligned X(\eta)&=\{\la=(\la_i)_{i\in\eta}\mid
\la_i\in\eta\},\quad\text{
and }\\
X^+(\eta)&=\{\la\in X(\eta)\mid \la_i\geq\la_{i+1},\,\,\forall
i\in\eta^\dashv\}\\
\endaligned
\end{equation}
 be the sets of weights and dominant weights, respectively.
Thus, both $\mbz^\iy$ and $\mbn^\iy$ are subsets of $X(\iy):=X(\mathbb Z)$. Recall
from \S2 that, for $i\in\mbz$, $\be_i=(\cdots,0,\underset
i1,0\cdots)\in\mbz^\iy$ and $\al_i=\be_i-\be_{i+1}\in\mbz^\iy$.
Let $\Pi(\iy)=\{\al_i\mid i\in\mbz\}$.
We introduce a
partial ordering $\leqwt$ on $X(\iy)$ by setting $\mu\leqwt\la$ if
$\la-\mu\in\mbz^+\Pi(\iy)$. For $\al=\sum_{i\in\mbz}k_i\al_i\in
\mbz\Pi(\iy)$ the number $\height(\al):=\sum_{i\in\mbz}k_i$ is
called the height of $\al$.

We first discuss the `standard' representation theory of $\iU$
which covers the category $\sC$ of weight modules, the
category\footnote{This category is denoted as $\sC'$ in
\cite{Lu93}.} $\intC$ of integrable weight modules, the category
$\sC^{hi}$, and the category $\sO$. All these categories are full
subcategories of the category $\iU$-$\bfMod$ of $\iU$-modules.

Let $M$ be a $\iU$-module. For $\la\in \iX$, let $M_\la=\{x\in
M\mid K_ix=\up^{\la_i}x\text{ for $i\in\mbz$}\}$. If
$M_\la\not=0$, then $\la$ is called a weight of $M$, $M_\la$ is
called a weight space and a nonzero vector in $M_\la$ is called a
weight vector of $M$. Let $\wt(M)=\{\la\in \iX\mid M_\la\not=0\}$
denote the set of weights of $M$. It is clear we have the
following lemma.
\begin{Lem}\label{weight}
Let $M$ be a $\iU$-module. Then $E_iM_\la\han M_{\la+\al_i}$ and
$F_iM_{\la}\han M_{\la-\al_i}$ for $i\in\mbz$.
\end{Lem}

A $\iU$-module $M$ is called a {\it weight module}, if
$M=\op_{\la\in\iX}M_\la$. Let $\sC$ denote the full subcategory of
$\iU$-$\bfMod$ consisting of all weight modules. It is clear every
submodule of a weight module and every quotient of a weight module
are weight modules. Hence, $\mc C$ is an abelian category.

An object $M\in\mc C$ is said to be integrable if, for any
 $x\in M$ and any $i\in\mbz$,
there exist $n_0\geq 1$ such that $E_i^nx=F_i^nx=0$ for any $n\geq
n_0$. Let $\intC$ be the full subcategory of $\mc C$ whose objects
are integrable $\iU$-modules.

Let $\mc C^{hi}$ be the full subcategory of $\mc C$ whose objects
are the $\iU$-module $M$ with the following property: for any
$x\in M$ there exist $n_0\geq 1$ such that $u_\cdot x=0$ where $u$
is a monomial in the $E_i$'s and $\deg(u)\geq n_0$.

For $\la\in \iX$ set $(-\iy,\la]=\{\mu\in \iX\mid \mu\leqwt\la\}$.
The category $\mc O$ is defined as follows. Its objects are
$\iU$-modules $M$ which are weight modules with finite dimensional
weight spaces and such that there exists a finite number of
elements $\la^{(1)},\cdots,\la^{(s)}\in \iX$ such that
$\wt(M)\han\bin_{i=1}^s(-\iy,\la^{(i)}]$.

\begin{Prop}\label{subcategory}
The category $\mc O$ is a full subcategory of $\mc C^{hi}$.
\end{Prop}
\begin{proof}Let $M$ be a object of $\mc O$. Then there exist $\la^{(i)}(1\leq i\leq s)$ such that
$\wt(M)\han\bin_{i=1}^s(-\iy,\la^{(i)}]$. Let $x_\mu\not= 0\in
M_\mu$. Let $n_\mu=\max\{\height(\la^{(i)}-\mu)\mid
\la^{(i)}\geqwt\mu,1\leq i\leq s\}$. We claim that if $k\geq
n_\mu+1$, then for any $i_1,\cdots, i_k\in\mbz$ we have
$\mu+\al_{i_1}+\al_{i_2}+\cdots+\al_{i_k}\not\in \wt(M)$. Indeed,
if there exists $1\leq j\leq s$ such that
$\mu+\al_{i_1}+\al_{i_2}+\cdots+\al_{i_k}\leq\la^{(j)}$, then
$\height(\la^{(j)}-\mu)\geq k\geq n_\mu+1$, contrary to the
definition of $n_\mu$. Thus, for any $k\geq n_{\mu+1}$ and any
$i_1,\cdots,i_k\in\mbz$, we have $E_{i_1}\cdots E_{i_k}x_\mu\in
M_{\mu+\al_{i_1}+\cdots+\al_{i_k}}=0$. Hence, $M$ is a object of
$\mc C^{hi}$.
\end{proof}

 A $\iU$-module $M$ is called a {\it highest
weight module} with highest weight $\la\in \iX$ if there exists a
nonzero vector $x_0\in M_\la$ such that $E_ix_0=0$,
$K_ix_0=\up^{\la_i}x_0$ for all $i\in\mbz$ and $\iU x_0=M$.  The
vector $x_0$ is called a highest weight vector. By a standard
argument (see, e.g., \cite{Lu89}) we have the following lemma.
\begin{Lem}Let $M$ be a highest weight $\iU$-module with highest weight $\la$. Then
$M=\bop\limits_{\mu\leq\la}M_\mu$ and $\dim M_\la=1$. Moreover $M$
contains a unique  maximal submodule.
\end{Lem}
For $\la\in \iX$, let
$$M(\la)=\iU/(\sum_{i\in\mbz}\iU E_i+
\sum_{i\in\mbz}\iU(K_i-\up^{\la_i}))=\bfU^-(\iy)1_\la,$$ where
$1_\la$ is the image of 1 in $M(\la)$. The module $M(\la)$ is
called a {\it Verma module}. By the above lemma, we know that
$M(\la)$ has a unique irreducible quotient module $L(\la)$. It is
clear, for $\la\in \iX$, the modules $M(\la)$ and $L(\la)$ are all
in the category $\mc C^{hi}$ and the category $\mc O$.

The next result classifies the irreducible modules in the
categories $\mc C^{hi}$ and $\mc O$.
\begin{Thm}\label{classfication for simple module in C^hi}
$(1)$ The map $\la\ra L(\la)$ defines a bijection between $\iX$ and the
set of isomorphism classes of irreducible $\iU$-modules in the
category $\mc C^{hi}$.

$(2)$ The map $\la\ra L(\la)$ defines a bijection between
$\iX$ and the set of isomorphism classes of irreducible $\iU$-modules in the category $\mc O$.
\end{Thm}
\begin{proof}Let $M$ be a irreducible $\iU$-module in the category
$\mc C^{hi}$. Choose a nonzero vector $x_\mu\in M_{\mu}$ for some
$\mu\in \wt(M)$. Let $$n_0=\max\{\deg(u)\mid \text{there exists
some monomial $u$ in the $E_i$'s such that  $u_\cdot
x_\mu\not=0$}\}.$$ Since $M\in\mc C^{hi}$, $n_0<\iy$. Let $u$ be
the monomial in the $E_i$'s  such that $\deg(u)\geq n_0$ and
$u_\cdot x_\mu\not=0$. Then $u_\cdot x_\mu$ is a highest weight
vector of $M$. Let $\la\in X(\iy)$ be the weight of $u_\cdot
x_\mu$. Then there exists a surjective homomorphism from $M(\la)$
to $M$. Since $L(\la)$ is a unique irreducible quotient module of
$M(\la)$ we have $M\cong L(\la)$, proving (1). The statement (2)
follows from (1) and \ref{subcategory}.
\end{proof}


\begin{Coro} An irreducible module $L(\la)$ in $\sC^{hi}$ is
integrable if and only if $\la\in X^+(\iy)$.
\end{Coro}

\begin{proof} If
$\la\in\iXp$, then the module $L(\la)$ is a homomorphic image of
the module $M(\la)/I(\la)$, where $I(\la)=\sum_{i\in\mbz}\iUm
F_i^{(\la_i-\la_{i+1}+1)}1_\la$. Note that $I(\la)$ is a proper
submodule of $M(\la)$. By \cite[3.5.3]{Lu93}, $M(\la)/I(\la)$ is
integrable. Hence, $L(\la)$ is integrable. Conversely, let
$L(\la)$ be a irreducible module in the category $\mc C^{hi}$, and
assume that $L(\la)$ is integrable. Then, there exists $x_0\in
M_{\la}$, $x_0\not=0$ and $E_ix_0=0$ for all $i\in\mbz$. By
\cite[(3.5.8)]{Lu93}, $\la\in \iXp$.
\end{proof}

Thus, using \ref{subcategory}, we have the following
classification theorem.

\begin{Thm}\label{classfication for simple module in C' cap C^hi}
$(1)$ The map $\la\ra L(\la)$ defines a bijection between $\iXp$
and the set of isomorphism classes of irreducible $\iU$-modules in
the category $\intC\cap\mc C^{hi}$.

$(2)$ The map $\la\ra L(\la)$ defines a bijection between $\iXp$
and the set of isomorphism classes of irreducible $\iU$-modules in
the category $\intC\cap\mc O$.
\end{Thm}

The integrable module $L(\la)$ has actually structure similar to
finite dimensional irreducible $\bfUn$-modules. Recall from
\cite[5.15]{Jan96} that, if $\mu\in\Xnnp$, then the irreducible
$\bfUn$-module $L(\mu)\cong\bfU([-n,n])/I_n(\mu)$ where
$$I_n(\mu)=\big(\sum_{-n\leq
i<n}(\bfU([-n,n])E_i+\bfU([-n,n])F_i^{\mu_i-\mu_{i+1}+1})+\sum_{-n\leq
j\leq n}\bfU([-n,n])(K_j-\up^{\mu_j})\big).$$

We want to prove that a similar isomorphism holds for $\iU$.

For any given $\la\in\iXp$, let $\ti {L(\la)}=M(\la)/I(\la)$, and
let $\lann=(\la_i)_{-n\leq i\leq n}\in\Xnnp$. Let
$\pi_{n,\la}:\bfUn\ra L(\lann)$ be the map sending $u$ to $u\bar
1_\la$, where $\bar 1_\la$ is the image of $1_\la$. Since
$I_n(\lann)\subset I_{n+1}(\la_{[-n-1,n+1]})$, there is a natural
$\bfUn$-module homomorphism $\iota_{n,\la}:L(\lann)\ra
L(\la_{[-n-1,n+1]})$ by sending $\pi_{n,\la}(u)$ to
$\pi_{n+1,\la}(u)$ for all $u\in\bfUn$, which is compatible with
the inclusion $\bfUn\hookrightarrow\bfU([-n-1,n+1])$. Thus, we
obtain a direct system $\{L(\lann)\}_{n\geq 1}$ whose direct limit
$\underset{\underset n\longrightarrow}\lim\,L(\lann)$ is naturally
a $\iU$-module . Further, since
$\iota_{n,\la}(\pi_{n,\la}(1))=\pi_{n+1,\la}(1)\not=0$, the map
$\iota_{n,\la}$ is injective. So we may identify $L(\lann)$ as a
$\bfUn$-submodule of $L(\la_{[-n-1,n+1]})$. Hence,
$\underset{\underset n\longrightarrow}\lim\,L(\lann)=\bin_{n\geq
1}L(\lann)$.

\begin{Thm}\label{L(la)}
We have, for $\la\in\iXp$,
$$L(\la)\cong\underset{\underset
n\longrightarrow}\lim\,L(\lann)\cong\ti{L(\la)}.$$
\end{Thm}
\begin{proof} Let $W$ be a $\iU$-submodule  of $\underset{\underset
n\longrightarrow}\lim\,L(\lann)$. Since $L(\lann)$ is irreducible $\bfUn$-module for $n\geq 1$ we have $W\cap L(\lann)=0$ or $W\cap L(\lann)=L(\lann)$. If there exist $n_0\geq 1$ such that $W\cap L(\lann)\not=0$ then $W\cap L(\lann)\not=0$ for $n\geq n_0$. Hence, $W\cap L(\lann)=L(\lann)$ for $n\geq n_0$. It follows that
$$W=\bin_{n\geq 1}(W\cap L(\lann))=\bin_{n\geq n_0}=(W\cap L(\lann))=\bin_{n\geq n_0}L(\lann)=\underset{\underset
n\longrightarrow}\lim\,L(\lann).$$ On the other hand, if   $W\cap
L(\lann)=0$ for $n\geq 1$ then $W=0$. Hence, $\underset{\underset
n\longrightarrow}\lim\,L(\lann)$ is an irreducible $\iU$-module.
Let $x_0:=\pi_{1,\la}(1)\in \underset{\underset
n\longrightarrow}\lim\,L(\lann)$. Then $x_0=\pi_{n,\la}(1)$ for
any $n\geq 1$. Hence, for $n\geq 1$ and $1\leq i< n$, $1\leq j\leq
n$ we have $E_ix_0=\pi_{n,\la}(E_i)=0$ and
$K_jx_0=\pi_{n,\la}(K_j)=\up^{\la_i}x_0$. Hence,
$\underset{\underset n\longrightarrow}\lim\,L(\lann)$ is a
irreducible highest weight module with highest weight $\la$. It
follows that $L(\la)\cong\underset{\underset
n\longrightarrow}\lim\,L(\lann)$. For any $n\geq 1$ there is a
natural $\bfUn$-module homomorphism $f_n$ from $L(\lann)$ to
$\ti{L(\la)}$ by sending $\pi_{n,\la}(u)$ to $\bar u$ for
$u\in\bfUn$. The maps $f_n$ ($n\geq 1$) induce a surjective
$\iU$-homomorphism $f$ from $\underset{\underset
n\longrightarrow}\lim\,L(\lann)$ to $\ti{L(\la)}$.  Since
$\underset{\underset n\longrightarrow}\lim\,L(\lann)$  is
irreducible $\iU$-module and $\Image(f)=\ti{L(\la)}\not=0$, $f$
has to be an isomorphism.
\end{proof}
\begin{Rem}
By \cite[4.4]{DP02} and \cite[\S6]{DD05} we can use the
PBW basis $\{A(\bfl)\mid A\in\Thim\}$  of $\iUm$ to define the canonical
basis $\{\frak c_{A}\mid A\in\Thim\}$ of $\iUm$. For $\la\in\iXp$, let $x_{\la}$
 be the highest weight vector of $L(\la)$. By \cite[8.10]{Lu90} and \ref{L(la)},
 one can easily show that the set $\{\frak c_Ax_\la\mid A\in\Thim\}\backslash\{0\}$
 forms a $\mbq(\up)$ basis of $L(\la)$.
\end{Rem}

Finally, it is interesting to point out that there are not
many {\it finite dimensional} weight $\bfU(\infty)$-modules.
For $n\geq 0$, let $J(n)$ be the two sided ideal of $\iU$
generated by $E_i$, $F_i$, $i\in (-\iy,-n)\cup[n,\iy)$.
\begin{Lem}\label{the two sided ideal J(infty)}
$(1)$ For $n\geq 0$, we have $J(n)=J(0)$.

$(2)$ Let $M$ be a finite dimensional $\iU$-module in the category $\mc
C$. Then $E_iM=F_iM=0$ for all $i\in\mbz$ and $\wt(M)\han \{\la\in
\iX\mid \la=k\bf 1\}$, where ${\bf 1}=(\cdots,1,1,\cdots,1,\cdots)\in
\iX$.
\end{Lem}
\begin{proof}Since $E_n,F_n\in J(n)$, by \ref{definition of
U(infty)}(e) we have $\ti K_n-\ti K_n^{-1}\in J(n)$. Hence,
$K_n^2-K_{n+1}^2\in J(n)$. By \ref{definition of U(infty)}(b) we
have $K_n^2E_{n-1}=\up^{-2}E_{n-1}K_n^2$ and
$K_{n+1}^2E_{n-1}=E_{n-1}K_{n+1}^2$. Hence,
$\up^{-2}E_{n-1}K_n^2-E_{n-1}K_{n+1}^2=(K_n^2-K_{n+1}^2)E_n\in
J(n)$. It follows that
$(\up^{-2}-1)E_{n-1}K_n^2=(\up^{-2}E_{n-1}K_n^2-E_{n-1}K_{n+1}^2)-E_{n-1}(K_n^2-K_{n+1}^2)\in
J(n)$. Hence, $E_{n-1}\in J(n)$. It is similar we can show
$F_{n-1},E_{-n},F_{-n}\in J(n)$. Hence, we have $J(n)=J(n-1)$ for
$n\geq 1$, and (1) follows.

Since $M$ is  finite dimensional weight module, we see that
$\wt(M)$ is a finite set. Assume $\wt(M)=\{\la^{(i)}\mid 1\leq
i\leq s\}$. By \ref{weight}  for $1\leq i\leq s$ there exist $n_i$
such that $E_jM_{\la^{(i)}}=F_jM_{\la^{(i)}}=0$ for
$j\not\in[-n_i,n_i)$. Let $n_0=\max\{n_i\mid 1\leq i\leq s\}$.
then $E_iM=F_iM=0$ for $i\not\in[-n_0,n_0)$. Hence, by (1) we have
$E_iM=F_iM=0$ for all $i\in\mbz$. Let $\la\in \wt(M)$ and $x_0$ be
a nonzero vector in $M_{\la}$. By \ref{definition of U(infty)}(e)
we have $\ti K_i^2x_0=\up^{2\la_i-2\la_{i+1}}x_0=0$. Hence,
$\la_i=\la_{i+1}$ for all $i\in\mbz$, proving (2).
\end{proof}
 Let $\sC^{fd}$ be the category of finite dimensional weight $\iU$-module.
Now by the above result we have the following classification of finite dimensional $\iU$-modules.
\begin{Thm} The modules $L(k{\bf 1})$ $(k\in\mbz)$ are
all non-isomorphic  finite dimensional irreducible $\iU$-module in
the category $\sC^{fd}$ and $\dim(L(k{\bf 1}))=1$. Moreover every
finite dimensional $\iU$-module in the category $\mc C$ is
complete reducible.
\end{Thm}

\section{Representations of $\iSr$}

In this section, we investigate the weight modules for $\iSr$ and
classify the irreducible ones.

 Recall from \S2 and \S5 the basis
$\{T_w\}_{w\in\frak S_r}$ for the Hecke algebra $\bfH$ and the
basis $\{\phi_{\la\mu}^w\}$  for the algebra
$\boldsymbol\sK(\iy,r)$ given in \eqref{phi basis}. Recall also
the various notations for the idempotents
$$\ttk_\la=[\diag(\la)]=\phi_{\la\la}^1,\qquad\la\in\Lair.$$ Let
\begin{equation}\label{varpi}
\vp=(\vp_i)_{i\in\mbz}\in\La(\iy,r),\,\,\text{ where
}\,\,\vp_i=\begin{cases}1,&\text{ if }1\leq i\leq r;\\
0,&\text{ otherwise.}\\\end{cases}
\end{equation}

We first observe the following which is clear from the definition
of the $\dagger$-completion $\tibfKir$ and the identification
$\iSr=\tibfKir$ (see \ref{iSr as dagger}).

\begin{Lem}\label{lemma2 for representation of U(infty,r)}
$(1)$ Let $\la,\mu\in\La(\iy,r)$. Then $\iSr\ttk_\la=\iUr\ttk_\la=
\bfKir\ttk_\la$, and
$\ttk_\la\iSr\ttk_\mu=\Hom_{\bfH}(x_\mu\bfH,x_\la\bfH)$.

$(2)$  The Hecke algebra $\bfH$ is isomorphic to
$\ttk_\vp\iSr\ttk_\vp$ by sending $T_u$  to $\phi_{\vp,\vp}^u$
($u\in\frak S_r$).

$(3)$ Let $\eta$ be a consecutive segment of $\mbz$. Then,
whenever $|\eta|\geq r$, $\bfH$ is isomorphic to a (centralizer)
subalgebra $e\boldsymbol\sS(\eta,r)e$ of $\boldsymbol\sS(\eta,r)$
for some idempotent $e$.
\end{Lem}

An $\iSr$-module $M$ is called a {\it weight module} if
$M=\oplus_{\la\in\La(\infty,r)}\ttk_\la M$. Clearly, \ref{lemma2
for representation of U(infty,r)}(1) implies that $\bfKir$ is a
weight $\iSr$-module. Let $\CrS$ be the category of weight
$\iSr$-modules. Since all $\ttk_\la\in\bfKir$, we can define
similarly the categories $\CrK$ and $\CrU$.

\begin{Prop}\label{weight module for U(infty,r)}
{\rm(1)} Up to category isomorphism, the categories $\CrK$, $\CrU$
and $\CrS$ are all the same.

{\rm(2)} The category $\CrS$ is a full subcategory of $\sC$.
\end{Prop}
\begin{proof} The statement (1) follows easily from part (1) of the lemma
above. We now prove (2). Using the homomorphism
$\zeta_r:\iU\to\iSr$, every $\iSr$-module $M$ is a $\iU$-module.
It suffices to prove that $\ttk_\la M=M_\la$ for all
$\la\in\La(\iy,r).$
 If $x\in M_\la$, then
$$\ttk_\la x=\prod_{i\in\mbz}\left(\left[{\ttk_i;0\atop\la_i}\right]x\right)=\prod_{i\in\mbz}
\left[{\la_i \atop\la_i}\right]x=x.$$ Hence,
$M_\la\subseteq\ttk_\la M$. Conversely, the inclusion $\ttk_\la
M_\la\subseteq M_\la$ follows from \ref{KKK}(3).
\end{proof}


Let $\bfOgir:=\iSr\ttk_\vp$. If we identify $\bfH$ with the
subalgebra $\ttk_\vp\iSr\ttk_\vp$ of $\iSr$ via the isomorphism
given in \ref{lemma2 for representation of U(infty,r)}(2), then
$\bfOgir=\bop\limits_{\mu\in\La(\iy,r)}\phi_{\mu\vp}^1\bfH$.
Hence, $\bfOgir$ is an ($\iSr$,$\bfH$)-bimodule.

\begin{Prop}\label{bimodule isomorphism}
 The evaluation map
 $$\text{\rm Ev}: \bfOgir\lra
 \bfOg_{\iy}^{\ot r},\,\, f\mapsto f(T_1)$$ defines an ($\iSr$,$\bfH$)-bimodule
isomorphism. Moreover, we  have
$$\bfH\cong\End_{\iSr}(\bfOg_{\iy}^{\ot
r})\cong\End_{\iU}(\bfOg_{\iy}^{\ot r}).$$
\end{Prop}
\begin{proof}
Since the set $\{\phi_{\mu\vp}^d\mid \mu\in\La(\iy,r),d\in\frak
D_\mu\}$ forms a basis for $\bfOgir$, and the set
$\{\phi_{\mu\vp}^d(T_1)\mid \mu\in\La(\iy,r),d\in\frak D_\mu\}$
forms a basis for $\bfOg_{\iy}^{\ot r}$. So the assertion follows
easily.
\end{proof}

 For $\la\in\Lair$, let $\la^t=(\la_i^t)_{i\in\mbz}\in\Lair$, where $\la_i^t=\#\{j\in\mbz|\la_j\geq i\}$ for $i\geq 1$ and $\la_i^t=0$ for $i\leq 0$.
 Let $w_\la$ be the unique element in $\fD_{\la,\la^t}$ such that $w_{\la}^{-1}\fS
 w_\la\cap\fS_{\la^t}=\{1\}$, and let $z_\la=\phi_{\la\vp}^1
T_{w_\la}y_{\la^t}$, where
$y_{\la^t}=\sum_{w\in\fS_{\la^t}}(-\up^2)^{-l(w)}T_w$.

For any $\mu\in\Lannr$, let
\begin{equation}\label{SpechtWeyl}
S^\mu=z_\mu\bfH,\quad\text{ and }\quad \iWmu=\iSr z_\mu
\end{equation}
 be the {\it Specht module} of $\bfH$ and the
{\it Weyl module} of $\iSr$, respectively. Note that, if
 $\mu\in\Lannr$ and $n\geq r$, then $\vp\in\Lannr$. So
  $$\nWmu:=\bfS([-n,n],r)z_\mu$$
  is well-defined. This is a Weyl module of $\bfS([-n,n],r)$. Since $\bfS([-n,n],r)$ is
a finite dimensional semisimple algebra, by \cite[4.6]{DJ91},
$\nWmu$ is an irreducible $\bfS([-n,n],r)$-module. Note that
$z_\mu=\ttk_\mu z_\mu$. So \ref{lemma2 for representation of
U(infty,r)} implies $\iWmu=\iSr\ttk_\mu z_\mu=\bfKir\ttk_\mu
z_\mu=\bfKir z_\mu.$  Hence, for any $\mu\in\La(\iy,r)$, we obtain
\begin{equation}\label{bin}
\iWmu=\bin_{n\geq r}\nWmu\cong\underset{\underset
n\longrightarrow}\lim\,\nWmu.
\end{equation}

Let $\La^+(r)$ be the set of all partitions of $r$. We will regard
$\La^+(r)$ as the subset $$\{\la\in\La(\iy,r)\mid
\la_i\geq\la_{i+1} \forall i\geq 1, \la_i=0\,\forall i\leq0\}$$ of
$\La(\iy,r)$. Define a map from $\La(\iy,r)$ to $\La^+(r)$ by
sending $\la$ to $\la^+$ where $\la^+$ is the unique element in
$\La^+(r)$ obtained by reordering the parts of $\la$.


\begin{Prop}\label{infinite Weyl module}
$(1)$ For $\mu\in\La(\iy,r)$ we have
$$\iWmu\cong\Hom_{\bfH}(S^\mu,\bfOg_{\iy}^{\ot
r})\cong\bop_{\la\in\La(\iy,r)}\Hom(S^\mu,x_\la\bfH).$$

$(2)$ Let $\mu\in\La(\iy,r)$. Then the module $\iWmu$ is an
irreducible $\iSr$-module. Moreover,  we have $\iWmu\cong
W(\iy,\mu^+)$.
\end{Prop}
\begin{proof}
The bimodule isomorphism Ev given in  \ref{bimodule isomorphism}
induces $\iSr$-module isomorphism
$$\Hom_\bfH(S^\mu,\bfOg_{\iy}^{\ot r})\cong
\Hom_\bfH(S^\mu,\bfOgir)
 \cong\bop_{\la\in\La(\iy,r)}\Hom(S^\mu,\phi_{\la\vp}^1\bfH).$$
Now \ref{lemma2 for non-surjective map} implies
 $$\Hom_\bfH(S^\mu,\bfOg_{\iy}^{\ot r})\cong \underset{\underset
n\longrightarrow}\lim\,\Hom_\bfH(S^\mu,\bfOg_{[-n,n]}^{\ot
r}).
$$
 Choose $m_0$ such that
$\mu\in\La([-m_0,m_0],r)$ and let $r_0=\max\{r,m_0\}$. Then, for
any $n\geq r_0$, we have, by \cite[8.1,8.7]{DJ91}, isomorphisms
$\Hom_\bfH(S^\mu,\bfOg_{[-n,n]}^{\ot r})\cong \nWmu$ which are
compatible with the direct systems. Now, the first isomorphism in
(1) is obtained by taking direct limits.

The assertions in (2) follow easily from \eqref{bin} (cf.
\cite[3.9]{DJ91}).
\end{proof}

 For
$\la,\mu\in\La(\iy,r)$ we write $\mu\leq_\wt^+\la$ if
$\mu^+\leq_\wt\la^+$. Note that $\mu^+\leq_\wt\la^+$ is equivalent
to the dominance order $\mu^+\trianglelefteq\la^+$ on partitions.

We also recall the notation of Young tableaux. Suppose
$\la\in\La(\iy,r)$ and $\mu\in\La^+(r)$. A $\mu$-tableau of type
$\la$ is a $\mu$-tableau with (possibly) repeated entries, where
for each $i$, the number of entries $i$ is equal to $\la_i$. We
denote the set of $\mu$-tableaux of type $\la$ by $\mcT(\mu,\la)$.
For $T\in \mcT(\mu,\la)$, we say that $T$ is row-standard (resp.,
strictly rwo-standard) if the numbers are weakly increasing
(resp., increasing) along each row of $T$. The column-standard and
strictly column-standard tableaux can be defined similarly. A
tableau $T$ is semistandard if it is row-standard and strictly
column-standard. Let $\mcT_0(\mu,\la)$ denote the set of
semistandard $\mu$-tableaux of type $\la$.

\begin{Prop}\label{weight space for infinite Weyl module} Let $\la\in\La(\iy,r)$ and $\mu\in\La^+(r)$.
 We have $\iWmu_\la\cong\Hom_{\bfH}(S^\mu,x_\la\bfH)$. Hence,
$\dim\iWmu_\la=\#\mcT_0(\mu,\la)$ (\cite[8.7]{DJ91}). In
particular, $\dim\iWmu_\mu=1$, and $\iWmu_\la\not=0$ implies
$\la\leq_\wt^+\mu$.
\end{Prop}
\begin{proof}Let $f\in\Hom_\bfH(S^\mu,x_\la\bfH)$. Assume $f(z_\mu)=x_\la
h_\la$ where $h_\la\in\bfH$. Then we have
\begin{equation*}
\begin{split}
(\ttk_if)(z_\la h)&=\ttk_if(z_\la h)=\ttk_i(x_\la h_\la h)=
\bigg(\sum_{\nu\in\La(\iy,r)}\up^{\nu_i}\phi_{\nu\nu}^1\bigg)(x_\la
h_\la
h)\\&=\sum_{\nu\in\La(\iy,r)}\up^{\nu_i}\phi_{\nu\nu}^1(x_\la
h_\la h)=\up^{\la_i}x_\la h_\la h=\up^{\la_i}f(z_\la h),
\end{split}
\end{equation*}
where $i\in\mbz$ and $h\in\bfH$. Hence, $\ttk_if=\up^{\la_i}f$ for
$i\in\mbz$. This means $f\in\iWmu_\la$. By \ref{infinite Weyl
module}(1), we obtain
 $(\iWmu)_\la=\Hom_{\bfH}(S^\mu,x_\la\bfH)$.
\end{proof}

Though the double centralizer property in the Schur-Weyl duality
is no longer true in the infinite case, the following
decomposition for the tensor space $\bfOg_{\iy}^{\ot r}$ continue
to hold.

\begin{Thm}\label{bimodule isomorphism for V(infty,r)}
We have the following ($\iSr$,$\bfH$)-bimodule isomorphism
 $$\bfOg_{\iy}^{\ot r}\cong\bop_{\mu\in\La^+(r)}\iWmu\ot S^\mu.$$
 Hence, as a (left) $\iSr$-module, $\bfOg_{\iy}^{\ot r}$ is
 completely reducible.
\end{Thm}
\begin{proof} It is well-known that, for all $n
\geq r$, there are ($\boldsymbol\sS([-n,n],r)$,$\bfH$)-bimodule
isomorphisms
$$\bfOg_{[-n,n]}^{\ot r}\cong\bop_{\mu\in\La^+(r)}W([-n,n],\mu)\ot S^\mu.$$
The require isomorphism is obtained by taking direct limits; cf.
\ref{lemma2 for non-surjective map} and \eqref{bin}.
\end{proof}

\begin{Coro}\label{lemma1 for classification of irreducible modle in Cr}
For each $\la\in\La(\iy,r)$, the $\iSr$-module $\iSr\ttk_\la$ is
completely reducible.
\end{Coro}
\begin{proof} Consider the $\iSr$-module homomorphism
$$f:\iSr\ttk_\la\lra\iSr\ttk_\vp\cong \bfOg_{\iy}^{\ot r},\,\,\,
u\ttk_\la\mapsto u\ttk_\la\phi_{\la\vp}^1\,\forall u\in\iSr.$$
Since $f$ is induced from the surjective map
$\phi_{\la\vp}^1:\bfH\to x_\la\bfH$, it follows that $f$ is
injective. Hence, $\iSr\ttk_\la$ is isomorphic to a submodule of
$\bfOg_{\iy}^{\ot r}$. Hence, the assertion follows from
\ref{bimodule isomorphism for V(infty,r)}.
\end{proof}
We are now ready to classify irreducible objects in the category
$\CrS$. We first observe the following.

\begin{Thm}\label{classfication for simple module in Cr}
$(1)$ The map $\mu\ra \iWmu$ defines a bijection between $\La^+(r)$ and
the set of isomorphism classes of irreducible $\iSr$-modules in
the category $\CrS$.

$(2)$ Any module in $\CrS$ is completely reducible.
\end{Thm}
\begin{proof}By  \ref{infinite Weyl module} and \ref{weight space for infinite Weyl
module}, the modules $\iWmu$ ($\mu\in\La^+(r)$) are irreducible
objects in $\CrS$. Let $M$ be an irreducible $\iSr$-module in
$\CrS$. Let $\mu\in \wt(M)$ and $x_0$ is a nonzero vector in
$M_\mu$. Then we have a surjective homomorphism $f$  from $\iSr$
to $M$ by sending $u$ to $ux_0$ for $u\in\iSr$. By the proof of
\ref{weight module for U(infty,r)}(2), we have $x_0=\ttk_\mu x_0$,
and so $f(\iSr\ttk_\mu)=\iSr f(\ttk_\mu)=\iSr\ttk_\mu x_0=\iSr
x_0=M$. Hence, $M$ is a quotient module of $\iSr\ttk_\mu$. By
 \ref{lemma1 for
classification of irreducible modle in Cr},  $M$ is isomorphic to
an irreducible component of $\iSr\ttk_\mu$, and hence, to an
irreducible component of $\bfOg_{\iy}^{\ot r}$. Now, \ref{bimodule
isomorphism for V(infty,r)} implies that $M$ is isomorphic to
$\iWmu$ for some $\mu\in\La^+(r)$, proving (1).

Let $N$ be an arbitrary module in $\CrS$.  Let $0\neq x\in N$.
Since $N\in\CrS$, we have $\iSr x=\bfKir x$ is a quotient module
of $\bfKir$. But,  by \ref{lemma2 for representation of
U(infty,r)},
$$\bfKir=\bop_{\la\in\Lair}\bfKir\ttk_\la=\bop_{\la\in\Lair}\iSr\ttk_\la.$$
 Hence, \ref{lemma1 for
classification of irreducible modle in Cr} implies that the
$\iSr$-module $\bfKir$ is completely reducible, and so $\iSr x$ is
completely reducible. Consequently,  $N$ is completely reducible.
\end{proof}
\begin{Rem}
Following the construction in \cite{Du92}, we can use the PBW type
basis $$\{[A]=\up^{-d_A}\phi_{\la\mu}^d\mid
A=\jmath(\la,d,\mu),(\la,d,\mu)\in{\frak D}(\infty,r)\}$$ for
$\bfKir$ to define the canonical basis $\{\theta_{\la\nu}^d\mid
(\la,d,\mu)\in{\frak D}(\infty,r)\}$ of $\bfKir$. By
\cite[5.3]{Du1992}  one can easily show that, for each
$\mu\in\La^+(r)$, the set $\{\theta_{\la\nu}^dz_{\mu}\mid
(\la,d,\nu)\in{\frak D}(\infty,r)\}\backslash\{0\}$ forms a
$\mbq(\up)$-basis for $\iWmu$. This basis is called the {\it
canonical basis} of $\iWmu$.
\end{Rem}
\section{Polynomial representations of $\iU$}

We are now ready to classify all irreducible polynomials
representations of $\iU$. Let $\polC$ be the full subcategory of
$\sC$ consisting of weight $\bfU(\infty)$-modules $M$ such that
$\wt(M)\han\mbni$. We call the objects in $\polC$ {\it polynomial
representations}. Let $r$ be a positive integer and let $\Cr$ be
the category of $\bfU(\infty,r)$-modules which are weight
$\iU$-modules via the homomorphism $\zeta_r:\iU\to\iUr$.

\begin{Prop}\label{lemma2 for classification of irreducible weight module for U(infty,r)}
$(1)$ The category $\Cr$ is a full subcategory of $\polC\cap\intC$.

$(2)$ For each $r\in\mbn$, there is a surjective homomorphism from
$\bfU(\infty,r+1)$ to $\iUr$ by sending the generators
$\tte_i,\ttf_i,\ttk_i$ for $\bfU(\infty,r+1)$ to the generators
$\tte_i,\ttf_i,\ttk_i$ for $\bfU(\infty,r)$, respectively. Hence,
$\Cr$ is a full subcategory of $\Cro$.
\end{Prop}
\begin{proof}Let $M$ be a $\iUr$-module in the category $\Cr$.
As a $\iU$-module, $E_i,K_i,F_i$ act on $M$ as the action of
$\tte_i,\ttk_i,\ttf_i$, respectively. But, by \ref{alp}, we have
$\tte_i^n=\ttf_i^n=0$ for all $n\geq r+1$. Hence, $M$ is an
integrable $\iU$-module.

We now prove $\wt(M)\han\mbzi$. Suppose this is not the case. Then
there exists $\la\in \wt(M)$ and $\la\not\in\mbzi$. Thus, $\la$
has an infinite support. Choose $r+1$ integers
$i_1<i_2<\cdots<i_{r+1}$  such that
$\la_{i_1}\not=0,\cdots,\la_{i_{r+1}}\not=0$ and
$\bft\in\La(\iy,r)$ such that $t_i=1$ for
$i\in\{i_1,\cdots,i_{r+1}\}$ and $t_i=0$ for all other $i$. Let
$u_\la$ be a nonzero vector in $M_\la$. Then
$$0\neq(\up^{\la_{i_1}}-1)\cdots
(\up^{\la_{i_{r+1}}}-1)u_\la=\prod_{i\in\mbz}[\up^{\la_i};t_i]^!u_\la=\prod_{i\in\mbz}[K_i;t_i]^!u_\la=\prod_{i\in\mbz}[\ttk_i;t_i]^!u_\la
\overset{\text{\ref{definition of S(infty,r)}}}=0u_\la=0,$$ a
contradiction.   Hence, $\wt(M)\han\mbzi$. On the other hand, by
\ref{definition of S(infty,r)} again, we have
$[\ttk_i;r+1]^!u_\la=(\up^{\la_i}-1)\cdots(\up^{\la_i}-\up^r)u_\la=0$
for $i\in\mbz$, forcing $0\leq\la_i\leq r$ for any $i\in\mbz$.
Hence, $\la\in\mbni$, proving (1).

Let $\Jr$ be the ideal of $\iU$ generated by
 $\prod_{i\in\mbz}[K_i;t_i]^!$, where $\bft=(t_i)_{i\in\mbz}\in\La(\iy,r+1)$.
By \ref{presentation and monomial base for U(infty,r)} we have
$\iUr$ is isomorphic to $\iU/\Jr$.  Let $\bft\in\La(\iy,r+2)$.
Choose $i_0\in\mbz$ such that $t_{i_0}>0$. Then
$\prod_{i\in\mbz}[K_i;t_i]^!=(K_{i_0}-\up^{t_{i_0}-1})(\prod_{i\not=i_0}[K_i;t_i]^!)[K_{i_0};t_{i_0}-1]^!\in
J_r$. Hence, $\Jrp\han\Jr$, and (2) follows.
\end{proof}

\begin{Lem}\label{lemma3 for classification of irreducible weight module for U(infty,r)}
Let $M$ be a irreducible $\iU$-module in the category $\polC$. Then $M\in\CrU$ for some $r\geq 0$.
\end{Lem}
\begin{proof}Let $\la\in \wt(M)$ and $u_\la$ be a nonzero vector in
$M_\la$. Then $M=\iU u_\la$ since $M$ is irreducible. Hence, by
\ref{weight}, we know $\sigma(\mu)=\sigma(\la)$ for any $\mu\in
\wt(M)$. Let $r=\sigma(\la)$. Since $\wt(M)\han\mbni$ we have
$\wt(M)\han\La(\iy,r)$. Thus, for any $\bft\in\La(\iy,r+1)$, there
exist $i_0\in\mbz$ such that $t_{i_0}>\la_{i_0}\geq 0$. Hence,
$[\up^{\la_{i_0}};t_{i_0}]^!=0$, and
$\prod_{i\in\mbz}[\ttk_i;t_i]^!u_\la=\prod_{i\in\mbz}[\up^{\la_i};t_i]^!u_\la=0$.
This shows that $M$ is a $\iUr$-module (by \ref{presentation and
monomial base for U(infty,r)}). Since $M\in\mc C$, it follows that
$M\in\CrU$.
\end{proof}
\begin{Thm}\label{classfication for simple module in tilde Cr}
$(1)$ For $r\in\mbn$ the modules  $\iWmu$ $(\mu\in\La^+(r'),r'\leq r)$
are  all non-isomorphic  irreducible $\iUr$-modules in the
category $\Cr$.

$(2)$ The modules  $\iWmu$ $(\mu\in\La^+(r), r\in\mbn)$ are  all
non-isomorphic  irreducible $\iU$-modules in the category $\polC$.
Moreover, there is only one irreducible $\iU$-module, the trivial
module $L(0)$, in the category $\polC\cap\mc C^{hi}$.
\end{Thm}
\begin{proof}By \ref{infinite Weyl
module}(2) and \ref{lemma2 for classification of irreducible
weight module for U(infty,r)}(2) we know $\iWmu$,
 where $\mu\in\La^+(r')$ with $r'\leq r$, are irreducible
$\iUr$-modules in $\Cr$. On the other hand, let $M$ be a
irreducible $\iUr$-module in the category $\Cr$. By \ref{lemma2
for classification of irreducible weight module for U(infty,r)}(1)
and \ref{lemma3 for classification of irreducible weight module
for U(infty,r)} we have $M\in\CrUp$ for some $r'\geq 0$. We claim
that $r'\leq r$. Indeed, suppose the contrary $r'>r$. Then, for
any $\la\in \wt(M)\han\La(\iy,r')$, there exists
$\mu\in\La(\iy,r+1)$ such that $\mu_i\leq\la_i$ for all
$i\in\mbz$. Thus, for any nonzero vector $x_0$ in $M_\la$,
$$\prod_{i\in\mbz}[K_i;\mu_i]^!x_0=\prod_{i\in\mbz}[\ttk_i;\mu_i]^!x_0=\prod_{i\in\mbz}[\up^{\la_i};\mu_i]^!x_0\not=0$$
for $i\in\mbz$. However, since $M$ is a $\iUr$-module, the
presentation of $\iUr$ implies
$\prod_{i\in\mbz}[K_i;\mu_i]^!x_0=0$, a contradiction. Hence,
$r'\leq r$. Now \ref{classfication for simple module in Cr}
implies that there exists $\mu\in\La^+(r')$ such that
$M\cong\iWmu$, proving (1).

The first assertion in (2) follows from \ref{classfication for
simple module in Cr} and \ref{lemma3 for classification of
irreducible weight module for U(infty,r)}.  By \ref{lemma2 for
classification of irreducible weight module for U(infty,r)},
$\iWmu$ is integrable. Since $\wt(\iWmu)\han\Lair$,  by the
classification of irreducible integral modules in $\sC^{hi}$ given
in \ref{classfication for simple module in C' cap C^hi}, we
conclude $\iWmu\not\in\sC^{hi}$ for any $r\geq 1$ and
$\mu\in\Lair$, proving the last assertion in (2).
\end{proof}
\begin{Rem}
(1) Classification of irreducible integrable modules over a
Kac-Moody algebra is an open problem (see \cite[Ex.
10.23]{Kac90}). However, irreducible integrable $\iUr$-modules can
be classified for all $r\in\mbn$. In fact, by the above result,
the modules  $\iWmu$ $(\mu\in\La^+(r'),r'\leq r)$ are  all
irreducible integrable $\iUr$-modules for any $r\in\mbn$.

It is clear that $\iSr\cong\prod_{\mu\in\Lair}\iSr\ttk_\mu$ as a
$\iSr$-module. Hence, $\iSr\not\in\sC$ as a $\iU$-module.  Since
$\iSr\ttk_\mu$ is a direct sum of finitely many irreducible
$\iSr$-modules in $\CrS$, $\iSr$ is the direct product of
irreducible $\iSr$-modules in $\CrS$.
\end{Rem}

(2) It is interesting to make a comparison between the
finite and infinite cases for quantum $\mathfrak{gl}_\eta$. In the
finite case, if we only consider finite dimensional
representations, then the categories $\sC^{hi}$ and $\intC$  are
the same as the category $\sC^{fd}$, and contain $\bS(n,r)\hmod$
as a full subcategory of polynomial representations of degree $r$
which more or less determine $\sC^{hi}$ and $\sC^{int}$. They are
all completely reducible and the irreducible modules in these
categories are all indexed by dominant weights. In the infinite
case, the situation is completely different. For example, the
categories $\sC^{hi}\cap\intC$, $\CrS$ and $\sC^{fd}$ become three
different categories. The complete reducibility continues to hold
in the last two categories, but seems not true in the first
category. The irreducible modules in $\sC^{hi}\cap\intC$ are also
indexed by dominant weights. Moreover, there are only a few
irreducible objects in $\sC^{fd}$, and there is only one
irreducible object in the category $\sC^{pol}\cap\sC^{hi}$.


\begin{thebibliography}{}\frenchspacing



\bibitem {BLM}
A.A. Beilinson, G. Lusztig and R. MacPherson, \textit{A geometric
setting for the quantum deformation of $GL_n$}, Duke Math.J. {\bf
61} (1990), 655-677.

\bibitem{Br}
J. Brundan, \textit{Kazhdan-Lusztig polynomials and character
formulae for the Lie superalgebra $\frak{gl}(m|n)$}, J. Amer.
Math. Soc. {\bf 16} (2003), 185-231.

\bibitem{DJKM}
M. Date, M. Jimbo, M. Kashiwara and T. Miwa,
\textit{Transformation groups for soliton equations}, Publ. RIMS,
Kyoto Univ. {\bf 18} (1982), 1077-1110.

\bibitem {DD05}
B. Deng and J. Du, \textit{Bases of quantized enveloping
algebras},  Pacific J. Math. {\bf 220}  (2005), 33-48.

\bibitem{DDPW} B. Deng, J. Du, B. Parshall and J. Wang, \textit{Finite
Dimensional Algebras and Quantum Groups}, Mathematical Surveys and
Monographs Volume 150, Amer. Math. Soc., Providence 2008.

\bibitem {DJ91}
R. Dipper and G. James,\ \textit{$q$-Tensor spaces and q-Weyl
modules}, Trans. Amer. Math. Soc. {\bf 327} (1991), 251-282.

\bibitem {DG}
S. Doty and A. Giaquinto, \textit{Presenting Schur algebras,}
International Mathematics Research Notices IMRN, {\bf 36} (2002),
1907-1944.

\bibitem {DGr}
S. R. Doty and R. M. Green, \textit{Presenting affine $q$-Schur
algebras}, Math. Z. {\bf 256} (2007), 311-345.


\bibitem {Du1992}
J. Du, \textit{Canonical bases for irreducible representations of
quantum $GL_n$}, Bull. London Math. Soc. {\bf 24} (1992), 325-334.


\bibitem {Du92}
J. Du, \textit{Kahzdan-Lusztig bases and isomorphism theorems for
$q$-Schur algebras,} Contemp. Math. {\bf 139} (1992), 121-140.

\bibitem{Du95}
J. Du, \textit{ A note on the quantized Weyl reciprocity at roots
of unity,} Alg. Colloq. {\bf 2}(1995), 363--372.

\bibitem {DFW}
J. Du, Q. Fu and J.-P. Wang, \textit{ Infinitesimal quantum
$\frak{gl}_n$ and  little $q$-Schur algebras,} J. Algebra {\bf
287} (2005), 199-233.

\bibitem{DP02}
J. Du and B. Parshall, \textit{Linear quivers and the geometric
setting of quantum ${\rm GL}\sb n$}, Indag. Math. {\bf 13} (2002),
459-481.


\bibitem {DP}
J. Du and B. Parshall, \textit{Monomial bases for $q$-Schur
algebras,} Trans. Amer. Math. Soc. {\bf 355} (2003), 1593-1620.

\bibitem{DPS3}
J. Du, B. Parshall and L. Scott, {\em Quantum Weyl reciprocity and
tilting modules}, Comm. Math. Phys. {\bf 195} (1998), 321-352.

\bibitem {DPW}
J. Du, B. Parshall, and Jian-pan Wang, \textit{Two-parameter
quantum linear groups and the hyperbolic invariance of $q$-Schur
algebras,} J. London Math. Soc. {\bf 44} (1991), 420-436.



\bibitem {Ha}
T. Hayashi, \textit{$Q$-analogues of Clifford and Weyl
algebras---spinor and oscillator representations of quantum
enveloping algebras}, Comm. Math. Phys. {\bf 127} (1990), 129-144.

\bibitem{JK}
G. James and A. Kerber, {\em The Representation Theory of the
Symmetric Group}, Encyclopedia of Math. and its Appl. no.~6,
Addison-Wesley, 1981 London.

\bibitem {Jan96}
J. C. Jantzen, \textit{Lectures on quantum groups}, Graduate
Studies in Mathematics, {\bf 6}. American Mathematical Society,
Providence, RI, 1996.

\bibitem {Ji}
M. Jimbo, \textit{A $q$-analogue of $U(\frak{gl}(N+1))$, Hecke
algebra, and the Yang-Baxter equation},  Lett. Math. Phys. {\bf
11}  (1986),  no. 3, 247-252.


\bibitem {Kac81}
V. G. Kac, and D. H. Peterson, \textit{Spin and wedge
representations of infinite-dimensional Lie algebras and groups},
Proc. Nat. Acad. Sci. U.S.A. {\bf 78} (1981), 3308-3312.

\bibitem {Kac87}
V. G. Kac and A. K. Raina, \textit{Bombay lectures on highest
weight representations of infinite-dimensional Lie algebras}
Advanced Series in Mathematical Physics, 2. World Scientific
Publishing Co., Inc., Teaneck, NJ, 1987.

\bibitem {Kac90}
V. G. Kac,  \textit{Infinite-dimensional Lie algebras}, Third
edition. Cambridge University Press, Cambridge, 1990.



\bibitem {LS91}
S. Levendorskii and  Y. Soibelman, \textit{Quantum group
$A_\infty$},  Comm. Math. Phys. {\bf 140}(1991),  399-414.





\bibitem {Lu89}
G. Lusztig, \textit{Modular representations and quantum groups,}
Comtemp. Math. {\bf 82} (1989), 59-77.

\bibitem {Lu901}
G. Lusztig, {\it Finite-dimensional Hopf algebras arising from
quantized universal enveloping algebra} J. Amer. Math. Soc. 3
(1990), 257-296.


\bibitem {Lu90}
G. Lusztig, \textit{Canonical bases arising from quantized
enveloping algebras},  J. Amer. Math. Soc. {\bf 3} (1990),
447-498.




\bibitem {Lu93}
G. Lusztig, \textit{Introduction to quantum groups}, Progress in
Math. {\bf 110}, Birkh$\ddot{\text{a}}$user, 1993.

\bibitem {Lu00}
G. Lusztig, \textit{Transfer maps for quantum affine
$\frak{sl}_n$}, In Representations and quantizations (Shanghai,
1998), China High. Educ. Press, Beijing (2000), 341-356.

\bibitem{Mc}
K. McGerty, \textit{Generalized $q$-Schur algebras and quantum
Frobenius}, Adv. Math. {\bf 214} (2007), 116-131.


\bibitem {MM}
K. Misra and T. Miwa, \textit{Crystal base for the basic
representation of $U_q(\widehat{\frak{sl}}(n))$}, Comm. Math.
Phys. {\bf 134} (1990), 79-88.


\bibitem {Pa1}
T. D. Palev, \textit{Highest weight irreducible unitary
representations of Lie algebras of infinite matrices. I. The
algebra ${\rm gl}(\infty)$},  J. Math. Phys. {\bf 31} (1990),
579-586.


\bibitem {Pa2}
T. D. Palev, \textit{Highest weight irreducible unitarizable
representations of Lie algebras of infinite matrices The algebra
$A_\infty$},  J. Math. Phys. {\bf 31} (1990), 1078-1084.


\bibitem {PS97}
T. D. Palev and N. I. Stoilova, \textit{Highest weight
representations of the quantum algebra $U_h(gl_\infty)$},  J.
Phys. A  {\bf 30} (1997), L699-L705.



\bibitem {PS98}
T. D. Palev and N. I. Stoilova, \textit{Highest weight irreducible
representations of the quantum algebra $U_h(A_\infty)$},  J. Math.
Phys. {\bf 39} (1998), 5832-5849.
\end{thebibliography}
\end{document}